\documentclass[hidelinks,onefignum,onetabnum,nohypdvips]{siamart250211}

\usepackage{tikz}
\usetikzlibrary{arrows.meta}
\usepackage{booktabs}
\usepackage{multirow}

\usepackage{lipsum}
\usepackage{amsfonts}
\usepackage{graphicx}
\usepackage{epstopdf}
\usepackage{algorithmic}
\usepackage{subfig}

\newsiamremark{remark}{Remark}
\newsiamremark{hypothesis}{Hypothesis}
\crefname{hypothesis}{Hypothesis}{Hypotheses}
\newsiamthm{claim}{Claim}
\newsiamremark{fact}{Fact}
\crefname{fact}{Fact}{Facts}

\headers{A Learnable Multigrid Framework}{Xiuping Wang and Li Luo}

\title{A Learnable Multigrid Framework via Graph Convolutions\thanks{Submitted to the editors DATE.
}}

\author{Xiuping Wang\thanks{Department of Mathematics, University of Macau, Macao SAR, China
  (\email{xiupingwang@um.edu.mo}, \email{liluo@um.edu.mo}).}
  \and Li Luo\footnotemark[2] \thanks{Corresponding author.}}

\usepackage{amsopn}

\begin{document}

\maketitle

\begin{abstract}
This paper presents a novel framework that integrates learnable graph convolutions 
with the geometric multigrid method for solving partial differential equations (PDEs).
The discretization of PDEs is first represented as a graph structure,
enabling the application of graph convolutions to enhance the multigrid performance.
By incorporating graph convolutions into the multigrid components such as smoothing and inter-grid transfer operators,
we develop a learnable multigrid that can adaptively optimize its performance based on the underlying problem characteristics.
In this framework, the graph convolutions are embedded directly within the multigrid cycle, effectively transforming the entire multigrid solver into a specialized neural network architecture, 
rather than combining a classical solver with surrogate models.
The learnable multigrid framework is lightweight in terms of parameter count 
and requires minimal training effort to achieve good performance.
Numerical experiments demonstrate the effectiveness of the proposed approach in solving some challenging problems,
showing improved convergence rates compared to traditional multigrid methods.
The generalizability of the learned parameters across different problem settings, including varying source terms, coefficients, geometries, and mesh sizes,
is also investigated with proper weight-sharing and transfer-learning strategies.

\end{abstract}



\section{Introduction}\label{sec:intro}
In this paper, we consider the following linear system 
\begin{equation}\label{eq:linear_system}
  A u = f,  
\end{equation}
where $A \in \mathbb{R}^{n \times n}$ is the system matrix, 
$u = [u_1, u_2, \dots, u_n]^T \in \mathbb{R}^n$ is the discrete solution vector,
and $n$ denotes the number of degrees of freedom.
The matrix $A$ typically arises from the discretization of partial differential equations (PDEs) 
using mesh-based methods such as the finite element method (FEM) \cite{babuvska2001finite,boffi2013mixed,luo2017parallel,Luo2019,russell1983finite,zienkiewicz1977finite}.
A general approach to solving the linear system \eqref{eq:linear_system} is to use iterative methods expressed as
\begin{equation}
  u^{m} = u^{m-1} + B (f - A u^{m-1}),
\end{equation}
where $B$ is an approximate inverse of $A$ and $m$ denotes the iteration step. The residual $r^{m}$ is defined as
\begin{equation}
  r^{m} = f - A u^{m}.
\end{equation}
Among various iterative methods, multigrid methods are widely used to solve large-scale problems efficiently \cite{bramble2019multigrid,hackbusch2013multi,xu2017algebraic,wesseling1995introduction}.
Multigrid methods include geometric multigrid, which exploits an explicitly constructed hierarchy of meshes and geometry-aware transfer operators, 
and algebraic multigrid, which constructs coarse-level operators and inter-grid transfers directly from the system matrix without relying on explicit geometric grids.

\smallskip
In this paper, we utilize the geometric multigrid method as our framework 
for simplicity of constructing the inter-grid transfer operators on the structured/unstructured hierarchical grids.
We make the multigrid solver learnable by introducing additional parameters and graph convolution operators in the smoothing and inter-grid transfer processes.
The required graph structure can be derived algebraically from the system matrix $A$ or geometrically from the underlying mesh used in the discretization process.
The solution $u$ and source term $f$ can be viewed as node features on the graph.

\smallskip
There have been some recent works combining multigrid methods with deep neural networks.
In 2019, He and Xu first proposed a unified framework of geometric multigrids and convolutional neural networks (CNN) \cite{he2016deep,krizhevsky2012imagenet,lecun2002gradient,simonyan2014very} named MgNet \cite{he_mgnet_2019}.
MgNet establishes a connection and partial equivalence between geometric multigrid methods and CNNs. 
It serves as an effective framework for both image classification and for recovering geometric multigrid methods to solve PDEs. 
MgNet has seen significant development in the context of image classification \cite{he_mgnet_2019, he_interpretive_2023, zhu_fv-mgnet_2023} and has recently been extended to MgNO \cite{he_mgno_2024,liu_mgno_2024} for operator learning tasks.
Other approaches utilize deep neural networks, particularly CNNs, to learn multigrid components such as smoothers and transfer operators on structured grids \cite{azulay2022multigrid,greenfeld2019learning,huang_learning_2021,katrutsa_deep_2017,lerer2024multigrid,xie2023mgcnn}.
For unstructured grids or algebraic multigrid contexts, Graph neural networks (GNNs) \cite{kipf2016semi,wu2020comprehensive} are frequently employed to enhance the solver performance. 
Some approaches focus on learning optimal parameters for classical algebraic multigrid components, such as interpolation weights or coarse grid selection \cite{luz2020learning,taghibakhshi2021optimization,yang_amgnet_2022,zou2023autoamg}. 
Others integrate GNNs directly into the iterative process to correct errors or improve smoothing on unstructured grids \cite{huang2024reducing,jiang_multigrid_2024}.
Beyond the learning approaches, evolutionary algorithms have been used to optimize multigrid solvers by employing genetic programming to construct non-standard cycles from existing components \cite{parthasarathy_automated_2026}.
Related ideas in learning-enhanced iterative methods have also been explored for nonlinear systems \cite{Gong2024,PINL}.

\smallskip
The main contributions of this paper are summarized as follows:
\begin{itemize}
	\item We present a novel learnable multigrid framework that integrates graph convolutions into the geometric multigrid method for solving PDEs.
	\item Our framework embeds graph convolutions directly within the multigrid cycle, 
	effectively transforming the entire multigrid solver into a specialized neural network architecture. 
	The graph convolutions are utilized in the smoothing and inter-grid transfer processes to enhance the solver's robustness and performance.
	\item The proposed approach is lightweight in terms of parameter count and computational cost, requiring minimal training samples and low offline training effort. 
	The learned parameters can be generalized to different problem settings, such as varying source terms, coefficients, geometries, and mesh sizes.
\end{itemize}
\smallskip
The rest of the paper is organized as follows:
In \cref{sec:preliminaries}, we provide preliminaries on the finite element method, geometric multigrid method, and graph neural networks.
\cref{sec:learnable-mg} introduces the proposed learnable multigrid framework with graph convolutions.
In \cref{sec:numerical}, we present numerical experiments to demonstrate the effectiveness of our approach.
Finally, \cref{sec:conclusion} concludes the paper and discusses future research directions.

\section{Preliminaries}\label{sec:preliminaries}
We briefly recall the graph representation of a discretized PDE by using the linear finite element method as an example, the classical geometric multigrid method, and the graph convolution used in the proposed solver.

\subsection{Graph representation}\label{subsec:graph-rep}
To use graph convolutions, we represent the degrees of freedom and their local couplings as a graph.
For a linear finite element discretization on a triangulation $\mathcal{T}_h$, the nodal basis functions $\{\phi_i\}_{i=1}^n$ give the expansion $u_h=\sum_{i=1}^n u_i\phi_i$.
The resulting stiffness matrix $A$ is sparse, and for standard scalar problems $A_{ij}\neq 0$ only when the corresponding vertices are connected by a mesh edge (or $i=j$).
This sparsity pattern naturally induces a graph.

\smallskip
For any triangulation $\mathcal{T}_h$,
we can construct a graph $\mathcal{G}_h = (\mathcal{V}_h, \mathcal{E}_h)$, where $\mathcal{V}_h$ is the set of nodes (vertices) $\{x_j\}_{j=1}^n$,
and the set of edges $\mathcal{E}_h$ is defined based on the connectivity of the mesh elements (see \cref{fig:mesh_to_graph}).
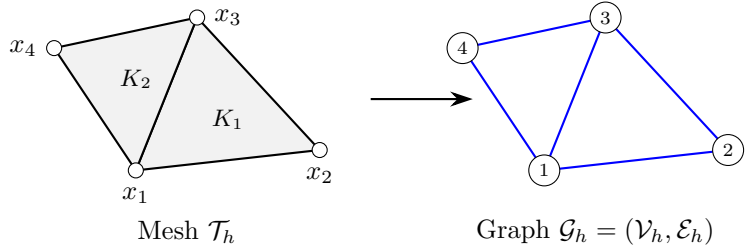
\begin{figure}[htbp]
	\centering
	\begin{tikzpicture}[scale=1.35,
		vertex/.style={circle, draw, fill=white, inner sep=2pt, font=\scriptsize},
		element_label/.style={font=\small}]
		
		\begin{scope}[xshift=-3.5cm]
			\coordinate (P1) at (0,0);
			\coordinate (P2) at (1.8,0.2);
			\coordinate (P3) at (0.6,1.5);
			\coordinate (P4) at (-0.8,1.2);
			
			\draw[thick, fill=gray!10] (P1) -- (P2) -- (P3) -- cycle;
			\draw[thick, fill=gray!10] (P1) -- (P3) -- (P4) -- cycle;
			
			\node[vertex, label=below:$x_1$] at (P1) {};
			\node[vertex, label=below:$x_2$] at (P2) {};
			\node[vertex, label=right:$x_3$] at (P3) {};
			\node[vertex, label=left:$x_4$] at (P4) {};
			
			\node[element_label] at (0.9, 0.5) {$K_1$};
			\node[element_label] at (0.0, 0.9) {$K_2$};
			\node at (0.5, -0.6) {Mesh $\mathcal{T}_h$};
		\end{scope}

		\draw[-{Stealth[length=3mm, width=2mm]}, thick] (-1.2, 0.7) -- (-0.2, 0.7);

		\begin{scope}[xshift=0.5cm]
			\coordinate (G1) at (0,0);
			\coordinate (G2) at (1.8,0.2);
			\coordinate (G3) at (0.6,1.5);
			\coordinate (G4) at (-0.8,1.2);
			
			\draw[thick, blue] (G1) -- (G2);
			\draw[thick, blue] (G2) -- (G3);
			\draw[thick, blue] (G3) -- (G4);
			\draw[thick, blue] (G4) -- (G1);
			\draw[thick, blue] (G1) -- (G3);
			
			\node[vertex] at (G1) {1};
			\node[vertex] at (G2) {2};
			\node[vertex] at (G3) {3};
			\node[vertex] at (G4) {4};
			
			\node at (0.5, -0.6) {Graph $\mathcal{G}_h=(\mathcal{V}_h, \mathcal{E}_h)$};
		\end{scope}
	\end{tikzpicture}
	\caption{Graph construction from an unstructured mesh. Left: The unstructured triangular mesh with vertices $x_i$ and elements $K_j$. Right: The corresponding graph representation where vertices become nodes and mesh edges become graph edges.}
	\label{fig:mesh_to_graph}
\end{figure}

Specifically, the node set and the edge set are defined as follows:
\begin{equation}
	\mathcal{V}_h = \{1, 2, \dots, n\}, \quad
	\mathcal{E}_h = \{(i, j) : x_i \text{ and } x_j \text{ share an edge in } \mathcal{T}_h\}.
\end{equation}
For graph convolutions we use directed edge pairs, so an undirected mesh edge is stored in both directions.
The neighbor set is $\mathcal{N}(i)=\{j:(i,j)\in\mathcal{E}_h\}$, the degree matrix is $D_g=\operatorname{diag}(|\mathcal{N}(i)|) \in \mathbb{R}^{n \times n}$, and the graph Laplacian is
\begin{equation}
	L_g=D_g-A_g, 
\end{equation}
where $A_g \in \mathbb{R}^{n \times n}$ is the graph adjacency matrix.
An edge weight function $w:\mathcal{E}_h\to\mathbb{R}$ may encode matrix entries or geometric information, depending on the chosen representation.

\smallskip
The construction can also be made algebraically from the system matrix, without using an explicit mesh:
\begin{equation}
	\mathcal{V} = \{1, \dots, n\}, \quad \mathcal{E} = \{(i, j) : A_{ij} \neq 0, i \neq j\}.
\end{equation}
The solution vector $u$ and right-hand side $f$ are then node features.
For scalar low-order finite elements, the mesh-connectivity graph and the matrix graph usually coincide.
For vector-valued PDEs, however, the geometric and algebraic viewpoints may lead to different graph representations.
In the geometric representation, each mesh vertex remains a single graph node and the multiple physical unknowns at that vertex can be stored as a multi-dimensional feature vector, for example $[u_i,v_i,\omega_i]^T$.
In the algebraic representation, each scalar degree of freedom in the global linear system can instead be treated as a separate graph node; for example, the unknowns $u_i$, $v_i$, and $\omega_i$ at the same mesh vertex may correspond to three distinct algebraic nodes.
Thus, the graph from the geometric viewpoint is often smaller with higher-dimensional node features, while the algebraic one is larger but follows the sparsity pattern of the assembled matrix directly.
For higher-order finite elements, a similar distinction appears because additional degrees of freedom may be associated with edges, faces, or elements rather than only with mesh vertices.

\smallskip
We mainly use geometric multigrid because the mesh hierarchy provides a simple way to construct the inter-grid transfer operators.
The main difference between geometric and algebraic multigrid lies in this construction: geometric multigrid builds the transfer operators from the mesh hierarchy, whereas algebraic multigrid constructs them from the algebraic structure of the system matrix.
Once the graph, smoother, prolongation, and restriction operators are fixed, the proposed framework can be applied to both geometric and algebraic multigrid.

\subsection{Classical geometric multigrid method}\label{subsec:classical-mg}
We consider the classical V-cycle geometric multigrid method to solve the linear system \eqref{eq:linear_system}.
The multigrid method consists of the following components:
\begin{itemize}
    \item Smoothing: Apply a few iterations of a smoother (e.g., Jacobi or Gauss-Seidel) to reduce high-frequency errors.
    \item Restriction: Transfer the residual from the fine grid to a coarser grid.
    \item Coarse Grid Correction: Solve the error equation on the coarse grid approximately/exactly.
    \item Prolongation: Interpolate the correction from the coarse grid back to the fine grid.
\end{itemize}
A single V-cycle iteration is summarized in \cref{alg:recursive-V-cycle}. We denote the number of levels by $J$,
with level $\ell = 1$ being the finest grid and level $\ell = J$ being the coarsest grid.
\begin{algorithm}[htbp]
\caption{V-cycle Multigrid Method: $u = \text{MG}(A_\ell, f, u, \ell)$}
\label{alg:recursive-V-cycle}
\begin{algorithmic}
\IF{$\ell = J$ }
	\STATE{Solve $A_J u = f$ exactly or approximately}
	\RETURN $u$
\ELSE
	\STATE{Pre-smoothing:} 
	\FOR{$j = 1$ to $\nu_\ell$}
	\STATE{$u \leftarrow u + B_\ell (f - A_\ell u)$}
	\ENDFOR
	\STATE{Restrict residual: $r_c = R_\ell (f - A_\ell u)$}
	\STATE{Initialize coarse correction: $e_c = 0$}
	\STATE{Recursive call: $e_c \leftarrow \text{MG}(A_{\ell+1}, r_c, e_c, \ell+1)$}
	\STATE{Prolongate correction: $e = P_\ell e_c$}
	\STATE{Update solution: $u \leftarrow u + e$}
	\STATE{Post-smoothing:}
	\FOR{$j = 1$ to $\nu_\ell$}
	\STATE{$u \leftarrow u + B_\ell (f - A_\ell u)$}
	\ENDFOR
	\RETURN $u$
\ENDIF
\end{algorithmic}
\end{algorithm}
Here, $B_\ell$ is the smoother at level $\ell$, $R_\ell$ is the restriction operator, $P_\ell$ is the prolongation operator, and $A_\ell$ is the system matrix at level $\ell$.
The subscript $c$ denotes the coarse grid quantities.
The number of smoothing steps $\nu_\ell$ can vary at different levels.

\smallskip
The efficiency of the geometric multigrid method relies on the complementary roles of the smoother and the coarse grid correction. 
Standard stationary iterative methods, such as Jacobi or Gauss-Seidel, act as smoothers that effectively damp high-frequency error components (oscillatory modes relative to the mesh size). 
However, they are inefficient at reducing low-frequency error components (smooth modes). 
The key insight of multigrid is that these smooth modes on a fine grid appear more oscillatory on a coarser grid, where they can be effectively eliminated by the smoothers. 
By recursively applying this strategy across the hierarchy, the method achieves mesh-independent convergence rates, ideally solving the system with $\mathcal{O}(n)$ computational complexity.

\smallskip
We denote the linear vector spaces associated with the finite element spaces on different levels as 
$V_{\ell}$ for $\ell = 1, 2, \dots, J$.
The inter-grid transfer operators connect the spaces $V_{\ell}$ and $V_{\ell + 1}$ associated with the fine and coarse meshes, respectively.
The prolongation operator $P_\ell: V_{\ell+1} \to V_{\ell}$ is typically defined by the natural inclusion of the coarse finite element space into the fine space.
The restriction operator $R_\ell: V_{\ell} \to V_{\ell+1}$ is usually chosen as the transpose of the prolongation operator (up to a scaling factor).
Furthermore, theoretical convergence proofs often require that the prolongation operator satisfies certain accuracy conditions, typically related to the order of the finite element space.

\smallskip
Regarding the coarse grid operators $A_\ell$ for $\ell > 1$, there are two standard approaches: geometric coarsening, where $A_\ell$ is assembled by discretizing the PDE directly on the coarse mesh $\mathcal{T}_{\ell}$, 
and Galerkin (or variational) coarsening, where the operator is defined algebraically as 
\begin{equation}
	  A_{\ell+1} = R_\ell A_\ell P_\ell.
\end{equation}
When Galerkin coarsening is employed, the choice of $R_\ell$ and $P_\ell$ is critical for the stability and convergence of the multigrid method. 
Specifically, for symmetric positive definite (SPD) problems, it is standard to choose $R_\ell = P_\ell^T$ to preserve the symmetry and positive definiteness of the coarse grid operators.

\smallskip
In the case of geometric coarsening, where $A_{\ell+1}$ is discretized directly on the coarse mesh, the relationship between the grid operators is less rigid than in the Galerkin case. 
Consequently, the restriction operator $R_\ell$ does not necessarily have to be the transpose of the prolongation operator $P_\ell$.
While this offers flexibility, it is crucial that the restriction and prolongation operators are sufficiently accurate to preserve the approximation properties required for the multigrid convergence.
Without this consistency, the coarse grid correction may fail to approximate the smooth error components accurately, leading to poor convergence or even divergence of the multigrid iteration.

\smallskip
The choice of coarsening strategies can be adapted based on the problem characteristics.
For instance, in problems with highly varying coefficients or anisotropies, the Galerkin approach may be preferred to better capture the operator's behavior across scales.
The geometric coarsening may be more straightforward for problems with uniform properties or when the mesh hierarchy is naturally defined.
However, as the Galerkin approach can lead to denser coarse grid operators, 
geometric coarsening may be favored for computational efficiency in certain scenarios.

\subsection{Graph neural networks}\label{subsec:gnns}
Graph neural networks are a class of neural networks designed to operate on graph-structured data.
They update node features by combining local information from each node and its neighbors.
In this paper, the only GNN operation we need is a graph convolution.
Let $\mathcal{G}=(\mathcal{V},\mathcal{E})$ be a graph with edge weights $w_{ij}$ and node features $X=(\boldsymbol{x}_i)_{i\in\mathcal{V}}$.
We define a graph convolution as a local map
\begin{equation}
	\Phi(X,\mathcal{G},w;\Theta)_i
	= U_\Theta\left(\boldsymbol{x}_i,
	\operatorname{AGG}_{j\in\mathcal{N}(i)}\left( M_\Theta(\boldsymbol{x}_i,\boldsymbol{x}_j,w_{ij})\right)\right), \,\, i \in \mathcal{V},
\end{equation}
where $\Phi(X,\mathcal{G},w;\Theta)_i$ is the updated feature at node $i$,
$M_\Theta$ is a learnable message function, $U_\Theta$ is a learnable update function,
$\Theta$ denotes the set of learnable parameters,
and $\operatorname{AGG}(\cdot)$ is a permutation-invariant aggregation such as summation or mean.
Thus, the output at each node depends only on its own feature, neighboring features, and edge weights.
For example, a one-layer linear graph convolution can be written as
\begin{equation}
	\Phi(X,\mathcal{G},w;\Theta)_i = W_0 \boldsymbol{x}_i + \sum_{j\in\mathcal{N}(i)} \widetilde{a}_{ij} W_1 \boldsymbol{x}_j,
\end{equation}
where $\widetilde{a}_{ij}$ is a normalized edge weight and $W_0,W_1$ are learnable weight matrices.
For PDE solvers, this local aggregation is analogous to relaxation (smoothing), such as the Jacobi iteration
\begin{equation}
	u_i^{m+1} = (1-\omega)u_i^{m} + \frac{\omega}{A_{ii}} (f_i - \sum_{j \neq i} A_{ij} u_j^{m}),
\end{equation}
where $\omega$ is the damping parameter. 
As we assume the system matrix $A$ is sparse, each update only involves the neighboring nodes defined by the nonzero entries of $A$.
The right-hand side $f$ can be incorporated as an additional feature in the GNN framework.

\smallskip
For GNNs, the loss function is typically defined based on the task at hand, such as node classification or graph regression.
For example, in a supervised learning setting, the loss is often defined as
\begin{equation}
	\mathcal{L} = \sum_{i \in \mathcal{V}_{\text{train}}} \ell(y_i, \hat{y}_i),
\end{equation}
where $\mathcal{V}_{\text{train}}$ represents the set of training nodes, $y_i$ denotes the ground truth label or target value, $\hat{y}_i$ is the predicted output from the network, and $\ell(\cdot, \cdot)$ is a metric measuring the discrepancy, such as cross-entropy or mean squared error.
Essentially, GNNs are data-driven approaches that learn to optimize the loss function through training on a dataset of graphs.

\smallskip
In the context of solving linear systems $A u = f$, a natural choice for the loss function is the residual norm:
\begin{equation}
	\mathcal{L} = \mathcal{L}(A, u, f) = \| f - A u \|_2.
\end{equation}

\smallskip
In multilevel GNNs, pooling and unpooling are analogous to restriction and prolongation in multigrid methods.
Pooling coarsens node features and graph connectivity to capture global information, while unpooling transfers coarse features back to the fine graph.
Together, these operations enable information aggregation over larger neighborhoods and redistribution of coarse-scale corrections.

\section{A learnable multigrid framework via graph convolutions}\label{sec:learnable-mg}
In this section, we propose a learnable multigrid framework that leverages graph convolutions 
to enhance the classical multigrid method.
The key idea is to integrate learnable components directly into the multigrid cycle,
transforming the entire multigrid solver into a specialized neural network architecture.
For instance, the correction is passed to the graph convolution after being applied to the smoother.
The update process for a stationary smoother can be expressed as:
\begin{align}
 \delta^j &= \theta B (f - A u^{j}), \quad j = 1, 2, \dots \\
	u^{j + 1} &= u^{j} + \eta \Phi_{B}(\delta^j, \mathcal{G}, w; \Theta_B),
\end{align}
where $\Phi_{B}$ is a graph convolution with learnable parameters $\Theta_B$, and $\theta$ and $\eta$ are learnable scaling factors.
$\Phi_{B}$ takes the node feature to be updated as its first input, followed by the graph structure $\mathcal{G}$ and edge weights $w$.
We write the learnable parameters explicitly; the same convention is used below for $\Phi_{R,\ell}$ and $\Phi_{P,\ell}$.
Additional inputs, such as node coordinates or other geometric features, can also be incorporated when needed.

\smallskip
This approach differs from the hybrid iterative methods in \cite{zhang2024blending,cui2025hybrid,hu2025hybrid}, where the classical smoother and the neural network-based smoother are applied sequentially:
\begin{equation}
	\begin{aligned}
	u^{(j, 0)} &= u^j, \quad j = 1, 2, \dots \\
	u^{(j, k)} &= u^{(j, k-1)} + B (f - A u^{(j, k-1)}), \quad k=1, \dots, \nu, \\
	u^{j + 1} &= u^{(j, \nu)} + \Psi_{B}(f - A u^{(j, \nu)}).
	\end{aligned}
\end{equation}
The hybrid iterative methods aim to apply deep neural networks as solvers to effectively reduce the error components that are not well-handled by the classical smoother.
The classical smoother is applied first to reduce high-frequency errors for one iteration or a few iterations. 
Then, the neural network-based solver is applied to further reduce the remaining error.
The $\Psi_{B}$ in the hybrid methods is typically a well-trained neural solver,
such as a DeepONet \cite{lu2019deeponet,lu2021learning} or a Fourier neural operator (FNO) \cite{li2020fourier,li2023fourier}.
In our proposed method, the one layer graph convolution is integrated directly into the smoothing step, as shown above.

\smallskip
For other smoothers like Krylov subspace methods such as GMRES \cite{saad1986gmres} and CG \cite{hestenes1952methods}, we can also combine them with graph convolutions similarly.
Take the GMRES smoother as an example.
The graph convolution is applied to generate the initial guess based on the residual for the Krylov subspace method.
The smoothing step can be expressed as:
\begin{equation}
	\begin{aligned}
		\delta_\mathrm{init} &= \theta \Phi_{B}(f - A u^j, \mathcal{G}, w; \Theta_B), \quad j = 1, 2, \dots \\
		u^{j+1} &= u^j + \eta\mathrm{GMRES}(A, f - A u^j, \delta_\mathrm{init}, v),
	\end{aligned}
\end{equation}
where $\mathrm{GMRES}(A, r, \delta_\mathrm{init}, v)$ denotes applying $v$ steps of the GMRES method to solve the linear system $A \delta = r$ with the initial guess $\delta_\mathrm{init}$.
A similar approach can be found in \cite{azulay2022multigrid}, where the initial guess is generated by a U-Net \cite{ronneberger2015u} architecture.

\smallskip
The restriction and prolongation operators are also combined with graph convolutions as follows:
\begin{align}
	r_c &= R_\ell \left(\Phi_{R,\ell}(r, \mathcal{G}, w; \Theta_{R,\ell})\right), \\
	e &= \Phi_{P,\ell}(P_\ell e_c, \mathcal{G}, w; \Theta_{P,\ell}),
\end{align}
where $\Phi_{R,\ell}$ and $\Phi_{P,\ell}$ are the graph convolutions for restriction and prolongation at level $\ell$, with learnable parameters $\Theta_{R,\ell}$ and $\Theta_{P,\ell}$, respectively.
A single V-cycle of the proposed learnable solver is summarized in \cref{alg:learnable-V-cycle}.

\begin{algorithm}[htbp]
	\caption{Learnable V-cycle: $u = \text{LMG}(A_\ell, f, u, \ell, \Theta)$}
	\label{alg:learnable-V-cycle}
	\begin{algorithmic}
	\IF{$\ell = J$ }
		\STATE{Solve $A_J u = f$ exactly or approximately}
		\RETURN $u$
	\ELSE
		\STATE{Pre-smoothing:} 
		\FOR{$j = 1$ to $\nu_\ell$}
		\STATE{$\delta = \theta_\ell B_\ell (f - A_\ell u)$}
		\STATE{$u \leftarrow u + \eta_\ell \Phi_{B, \ell}(\delta, \mathcal{G}_\ell, w_\ell; \Theta_{B,\ell})$}
		\ENDFOR
		\STATE{Compute residual: $r = f - A_\ell u$}
		\STATE{Restrict residual: $r_c = R_\ell \Phi_{R,\ell}(r, \mathcal{G}_\ell, w_\ell; \Theta_{R,\ell})$}
		\STATE{Initialize coarse correction: $e_c = 0$}
		\STATE{Recursive call: $e_c \leftarrow \text{LMG}(A_{\ell+1}, r_c, e_c, \ell+1, \Theta)$}
		\STATE{Learnable prolongation refinement: $e = \Phi_{P,\ell}( P_\ell e_c, \mathcal{G}_\ell, w_\ell; \Theta_{P,\ell})$}
		\STATE{Update solution: $u \leftarrow u + e$}
		\STATE{Post-smoothing:}
		\FOR{$j = 1$ to $\nu_\ell$}
		\STATE{$\delta = \theta_\ell B_\ell (f - A_\ell u)$}
		\STATE{$u \leftarrow u + \eta_\ell \Phi_{B, \ell}(\delta, \mathcal{G}_\ell, w_\ell; \Theta_{B,\ell})$}
		\ENDFOR
		\RETURN $u$
	\ENDIF
	\end{algorithmic}
\end{algorithm}
The learnable multigrid framework can be interpreted as a GNN architecture.
The V-cycle structure defines the overall architecture of the network.
Each layer of the network corresponds to a level in the multigrid hierarchy.
The graph convolutions are embedded into each layer as learnable components.
Besides the graph convolutions, the scaling factors $\theta_\ell$ and $\eta_\ell$ at different levels are also learnable parameters.
We summarize all the learnable parameters in the proposed framework as
\begin{equation}
	\Theta = \{ \theta_\ell, \eta_\ell, \Theta_{B,\ell}, \Theta_{R,\ell}, \Theta_{P,\ell} : \ell = 1, 2, \dots, J-1 \}.
\end{equation}
If the coarsest-level problem is solved approximately by a learnable solver rather than exactly, then the corresponding learnable parameters at level $J$ should also be included in $\Theta$.

\smallskip
In this paper, we utilize only one layer graph convolution for each $\Phi_{B,\ell}$, $\Phi_{R,\ell}$, and $\Phi_{P,\ell}$,
which can also be replaced by more advanced architectures with multiple layers and non-linear activation functions as needed.
The advanced architectures may further improve the performance with more parameters.
However, they need to be carefully trained with sufficient training data.
A well-trained model with more parameters may have better performance and generalization ability than the classical smoother and inter-grid transfer operators.
For instance, the Incomplete Lower-Upper (ILU) preconditioner is completely replaced by a well-trained GNN in \cite{yusuf2024constructing}.
We use one layer graph convolutions in this work to keep the number of learnable parameters low 
and reduce the offline training costs and difficulty.

\smallskip
We can also leave some of the components as classical ones without learnable parts.
If the classical components are optimal or effective enough,
we can keep them unchanged to reduce the number of learnable parameters and computational costs as the graph convolutions can be computationally expensive for large-scale problems.
For example, if the Jacobi smoother with an optimal damping factor is given for a problem, 
it can be used directly without learnable graph convolutions.
The proposed framework is flexible, allowing for easy integration of different types of graph convolutions and architectures.

\subsection{Training as an iterative solver}\label{subsec:training}
The training process mimics the iterative solving process and is done in an unsupervised learning manner.
Given the system matrix $A$ and the source term $f$, the training starts with an initial guess $u^0$ (e.g., a zero vector).
At the $m$-th iteration, the current solution $u^m$ is updated by applying one V-cycle of the learnable multigrid solver:
\begin{equation}
	u^{m+1} = \text{LMG}(A, f, u^m, 1, \Theta).
\end{equation}
The training aims to minimize the loss function (residual norm) by updating the learnable parameters $\Theta$ using backpropagation.
One training epoch consists of one forward V-cycle iteration of the learnable multigrid solver plus one backpropagation step for updating $\Theta$.
The training stops when the loss function is less than or equal to a predefined tolerance or reaches the maximum number of epochs.

\smallskip
The training process can be summarized in \cref{alg:training-process}.
The unsupervised training does not require labeled data, as the loss function is defined based on the residual of the linear system itself.
The model learns to minimize the residual norm directly through the multigrid iterations.
We choose the Rprop (Resilient Backpropagation) optimizer \cite{riedmiller1993rprop} for updating the parameters in our experiments,
which is a gradient-based optimization algorithm that adapts the weight updates based on the sign of the gradients.
The number of parameters in the proposed framework is relatively small,
making Rprop a suitable choice for optimization.
Compared to the widely used Adam optimizer \cite{kingma2014adam}, 
Rprop often shows better performance in our experiments for training the learnable multigrid solver.
\begin{algorithm}[htbp]
\caption{Training Process}
\label{alg:training-process}
\begin{algorithmic}
\STATE{Given: System matrix $A$, right-hand side vector $f$, initial parameters $\Theta$, optimizer, max epochs $N_{max}$, tolerance $tol$}
\STATE{Initialize $u^0 = 0$}
\FOR{$m = 1$ to $N_{max}$}
	\STATE{Forward Pass:}
	\STATE{\quad $u^m = \text{LMG}(A, f, u^{m-1}, 1, \Theta)$}
	\STATE{Compute Loss:}
	\STATE{\quad $\mathcal{L} = \| f - A u^m \|_2$}
	\IF{$\mathcal{L} \leq tol$}
		\STATE{Break}
	\ENDIF
	\STATE{Backward Pass:}
	\STATE{\quad Compute gradients $\nabla_\Theta \mathcal{L}$ via backpropagation}
	\STATE{Parameter Update:}
	\STATE{\quad Update $\Theta$ using the optimizer}
\ENDFOR
\end{algorithmic}
\end{algorithm}

\smallskip
The trained model with parameters $\Theta$ can then be used as an iterative solver or a preconditioner for inference.
We expect the best performance for the inference process when the model follows the same V-cycle structure as in the training phase.
The model can also be different from the training phase, such as using more smoothing steps or more levels in the multigrid hierarchy.

\subsection{Graph convolution operators}\label{subsec:gcn-ops}
The graph convolution operators $\Phi_{B,\ell}$, $\Phi_{R,\ell}$, and $\Phi_{P,\ell}$ can be chosen from various GNN architectures.
In this paper, we utilize the SAGEConv \cite{hamilton2017inductive} as the graph convolution operator.
The SAGEConv operator aggregates neighbor information using mean, LSTM \cite{hochreiter1997long}, or max pooling methods.
In this work, we use the mean pooling method for aggregation.
The update rule of the SAGEConv for node $i$ is given by:
\begin{equation}
	\boldsymbol{x}_i = W_1 \cdot \boldsymbol{x}_i + W_2 \cdot \text{mean}_{j \in \mathcal{N}(i)} \boldsymbol{x}_j,
\end{equation}
where $\boldsymbol{x}_i$ is the features of node $i$, and $W_1, W_2$ are learnable weight matrices.

\smallskip
There exist various graph convolution operators, such as ChebConv \cite{defferrard2016convolutional}, GraphConv \cite{morris2019weisfeiler}, etc.
The choice of graph convolution can be flexible and adapted to different problems.
As we aim for solving PDEs rather than classification tasks,
we choose the relatively simple graph convolution operators and recommend the SAGEConv.
We also find that the ChebConv operator also shows good performance in our experiments,
but it requires more computational resources due to the higher complexity of the operator.

\subsection{Spectral analysis of the learnable smoother}
\label{subsec:spectral-analysis}
In this subsection, we analyze the learnable smoother from a spectral perspective
when the SAGEConv operator with mean pooling and the damped Jacobi method are used.
For simplicity, we consider the scalar case where the feature dimension is $d = 1$.
In this case, the weight matrices $W_1$ and $W_2$ reduce to scalars, which we denote by $w_1$ and $w_2$, respectively.
For vector-valued features with $d > 1$, the analysis can be extended by considering the spectral properties of the weight matrices and their interactions with the graph structure.

\smallskip
The SAGEConv applied to the update $\delta \in \mathbb{R}^n$ from the classical smoother can be written in matrix form as
\begin{equation}
	\label{eq:graph-conv-matrix}
	\Phi_B(\delta) = w_1\, \delta + w_2\, D_g^{-1} A_g\, \delta = \left(w_1 I + w_2 D_g^{-1} A_g\right) \delta.
\end{equation}
The learnable smoothing iteration becomes
\begin{equation}
	\label{eq:learnable-smoother-matrix}
	u^{m+1} = u^{m} + \eta\theta\, \left(w_1 I + w_2 D_g^{-1} A_g\right) D_A^{-1}\,(f - Au^m),
\end{equation}
where $D_A$ is the diagonal of the system matrix $A$.
The corresponding error propagation matrix is
\begin{equation}
	\label{eq:error-propagation}
	S = I - \eta\theta\, \left(w_1 I + w_2 D_g^{-1} A_g\right) D_A^{-1} A.
\end{equation}

\begin{remark}[Recovery of classical methods]
	By setting $w_1 = 1$, $w_2 = 0$, $\theta = \omega$, and $\eta = 1$,
	the error propagation matrix \cref{eq:error-propagation} reduces to 
	$S = I - \omega D_A^{-1} A$, 
	which is the iteration matrix of the damped Jacobi method with damping factor $\omega$.
	Hence, the learnable smoother generalizes the classical Jacobi smoother.
\end{remark}

\begin{remark}[Connection to polynomial smoothers]
	\label{rmk:polynomial-smoother}
By noticing that $D_g^{-1}A_g = I - D_g^{-1}L_g$, 
we can rewrite the effective preconditioner as
\begin{equation}
	\left((w_1 + w_2) I - w_2\, D_g^{-1}L_g\right) D_A^{-1},
\end{equation}
which is a degree-1 polynomial in the normalized graph Laplacian $D_g^{-1}L_g$ applied to the classical smoother $B = D_A^{-1}$.
This is analogous to classical polynomial smoothers, such as Chebyshev smoothers \cite{hackbusch2013multi}, where the smoother applies a polynomial in the operator to accelerate error reduction.
The learnable parameters $w_1$ and $w_2$ play the role of optimizable polynomial coefficients.
Furthermore, using a ChebConv operator of order $K$ in place of SAGEConv would yield a degree-$K$ polynomial smoother, 
providing a spectral interpretation for the choice of graph convolution architecture.
\end{remark}

\smallskip
The spectral analysis reveals that for the scalar case, the SAGEConv-based learnable smoother introduces a frequency-dependent damping.
Consider the case where the eigenvectors of $D_A^{-1} A$ and $D_g^{-1}A_g$ are known.
Let $\{\lambda_k\}$ and $\{\hat{\lambda}_k\}$ denote the eigenvalues of $D_A^{-1} A$ and $D_g^{-1}A_g$, respectively.
When both operators share a common eigenbasis (e.g., for the Poisson equation on a uniform grid, where $D_A^{-1}A$ and $D_g^{-1}A_g$ are simultaneously diagonalizable by the discrete Fourier modes),
the eigenvalues of the error propagation matrix $S$ are
\begin{equation}
	\label{eq:eigenvalue-S}
	\sigma_k = 1 - \eta\theta(w_1 + w_2\hat{\lambda}_k)\lambda_k.
\end{equation}
For the classical damped Jacobi method ($w_1 = 1$, $w_2 = 0$), these reduce to 
\begin{equation}\label{eq:eigenvalue-jacobi}
	\sigma_k = 1 - \omega\lambda_k,
\end{equation}
and the smoothing performance is controlled by a single parameter $\omega$.

\smallskip
With $w_2 \neq 0$, the factor $(w_1 + w_2\hat{\lambda}_k)$ introduces a frequency-dependent damping.
Since the eigenvalues $\hat{\lambda}_k$ of $D_g^{-1}A_g$ vary across modes, 
different eigenmodes will receive different effective damping factors.
While in the classical Jacobi method \cref{eq:eigenvalue-jacobi}, the damping is uniform across all modes.
This is the main reason why the learnable smoother can potentially achieve better performance than the classical smoother.

\begin{proposition}
	\label{prop:optimal-smoothing}
	Let $S_\omega = I - \omega D_A^{-1}A$ be the error propagation matrix of the damped Jacobi method with optimal damping factor $\omega^*$.
	Let $S_{w_1, w_2} = I - \eta\theta(w_1 I + w_2 D_g^{-1}A_g) D_A^{-1}A$ be the error propagation matrix of the learnable smoother.
	Then
	\begin{equation}
		\min_{w_1, w_2, \eta, \theta} \rho(S_{w_1, w_2}) \le \min_{\omega} \rho(S_\omega) = \rho(S_{\omega^*}),
	\end{equation}
	where $\rho(\cdot)$ denotes the spectral radius.
	That is, the optimal learnable smoother achieves a smoothing factor at least as good as the optimal damped Jacobi method.
\end{proposition}
\begin{proof}
	The classical damped Jacobi iteration matrix $S_\omega = I - \omega D_A^{-1}A$ is recovered by setting $w_1 = 1$, $w_2 = 0$, $\theta = \omega$, and $\eta = 1$ in $S_{w_1, w_2}$.
	Since the parameter space of the learnable smoother $(w_1, w_2, \eta, \theta)$ strictly contains the space of the damped Jacobi method (parameterized by $\omega$ alone), the minimum over the larger space cannot exceed the minimum over the smaller space.
\end{proof}

\begin{remark}
\Cref{prop:optimal-smoothing} guarantees that introducing the graph convolution into the smoother cannot degrade the best achievable smoothing factor.
In practice, the additional degree of freedom provided by $w_2$ allows the optimizer to learn frequency-dependent damping,
which can lead to strict improvement over the classical Jacobi smoother, particularly for problems where a uniform damping factor is suboptimal.
\end{remark}

\subsection{Initialization of learnable parameters}\label{subsec:init}
The initialization of the learnable parameters $\Theta$ can significantly impact the training process and convergence of the learnable multigrid solver.
We aim to improve the classical multigrid method by integrating learnable components.
Therefore, we do not want to deviate too far from the classical components at the beginning of the training.

\smallskip
For the graph convolution operators $\Phi_{B,\ell}$, $\Phi_{R,\ell}$, and $\Phi_{P,\ell}$,
we initialize the weight matrices to be identity mappings.
For instance, for the SAGEConv operator, we set
\begin{equation}
	W_1^{(\ell)} = I_d, \quad W_2^{(\ell)} = 0,
\end{equation}
where $I_d$ is the identity matrix of the feature dimension $d$.

\smallskip
For the scaling factors $\theta_\ell$ and $\eta_\ell$ at different levels,
we initialize them to the same values used in the classical multigrid method.
For example, if the Jacobi smoother with a damping factor $\omega$ is used,
we set $\theta_\ell = \omega$ and $\eta_\ell = 1$.

\smallskip
This initialization strategy ensures that the learnable multigrid solver starts identical to the classical multigrid method.
At the first iteration of training, the solver behaves exactly like the classical multigrid method.
From the spectral analysis in \cref{subsec:spectral-analysis}, this identity initialization corresponds to recovering the classical damped Jacobi smoother
(cf.\ \cref{eq:error-propagation} with $w_1 = 1$, $w_2 = 0$).
Because the classical multigrid is already a convergent iteration with $\rho(S) < 1$,
the training starts from a well-behaved point in the parameter space
rather than from a random, potentially divergent configuration.
Subsequent gradient updates then explore the enlarged parameter space to further reduce the spectral radius of $S$, as guaranteed by \cref{prop:optimal-smoothing}.

\section{Numerical experiments}
\label{sec:numerical}
In this section, we present numerical experiments to demonstrate the performance of the proposed learnable multigrid solver.
The mesh size of the finest grid is denoted as $h_{\min}$ in the following sections.
We set the maximum number of epochs as 100. 
The training stops when the loss function is less than or equal to the tolerance 
\[
tol = \frac{10^{-6}}{\mathcal{L}^0}
\] 
or reaches the maximum number of epochs. 
The learning rate is set to 0.01 for all the experiments.
The coarsest-level problem is solved approximately in all the experiments.
All the numerical experiments are implemented in Python using PyTorch \cite{paszke2019pytorch} and PyTorch Geometric \cite{fey2019fast} with double precision.
The experiments are conducted on a Macbook Pro with a 24GB Apple M4 Pro chip.

\subsection{Darcy equation}\label{subsec:darcy}
We first consider the following Darcy equation for single-phase flow in porous media as our model problem:
\begin{equation}\label{eq:darcy}
	\begin{aligned}
		&\nabla \cdot \boldsymbol{u} = f, \quad \text{in } \Omega, \\
		&\boldsymbol{u} = -\boldsymbol{K} \nabla p, \quad \text{in } \Omega, \\
		&p = g_D, \quad \text{on } \Gamma_D, \\
		&\boldsymbol{u} \cdot \boldsymbol{n} = g_N, \quad \text{on } \Gamma_N,
	\end{aligned}
\end{equation}
where $\boldsymbol{K}$ is the conductivity tensor, $p$ is the pressure, 
$\boldsymbol{u}$ is the Darcy velocity, $f$ is the source term, and $\Omega$ is the computational domain.
The boundary $\partial \Omega$ is divided into two disjoint parts: $\Gamma_D$ and $\Gamma_N$,
where Dirichlet and Neumann boundary conditions are applied, respectively.
We consider the staggered grid finite difference method on a uniform grid to discretize the Darcy equation \cref{eq:darcy}.
The mesh size of the finest grid is set as $h_{\min} = 1/100$.
The number of levels $J$ is set as 7.

\smallskip
The graph structure for the graph convolutions in this example is purely built from the algebraic structure of the system matrix $A_\ell$.
This means that the graph $\mathcal{G}_\ell$ at level $\ell$ has nodes corresponding to the degrees of freedom at that level, and edges corresponding to the nonzero entries in $A_\ell$.

\smallskip
Generally, the conductivity $\boldsymbol{K}$ is highly heterogeneous in porous media applications.
Thus, we consider the Galerkin coarsening to construct the coarse grid operators $A_\ell$ for $\ell < J$.
All the different levels share the same restriction and prolongation graph convolutions, i.e.:
\begin{equation}
	\Phi_{R,\ell} = \Phi_{R}, \quad \Phi_{P,\ell} = \Phi_{P}, \quad \ell = 1, 2, \dots, J-1.
\end{equation}
We set the damping factor as $\omega = 2/3$ and the number of smoothing steps as $\nu_\ell = 4$.
We build two variants of the learnable multigrid solver, both using piecewise constant interpolation for the prolongation operator $P_\ell$ and its transpose for the restriction operator $R_\ell$:
\begin{itemize}
	\item \textbf{LMG-RP}: Only the restriction and prolongation graph convolutions $\Phi_R$ and $\Phi_P$ are trained. The smoother is kept as the classical damped Jacobi method without learnable graph convolutions.
	\item \textbf{LMG-RPS}: The restriction and prolongation graph convolutions $\Phi_R$, $\Phi_P$, the smoother graph convolutions $\Phi_{B,\ell}$ and the scaling factors $\eta_\ell$ at each level are trained. The damping coefficient $\theta_\ell$ is fixed during the training process.
\end{itemize}
For comparison, we also implement a classical multigrid solver with bilinear interpolation.
Essentially, we aim to improve the convergence of the piecewise constant multigrid solver through learned graph convolutions
and compare both trained variants with the classical bilinear interpolation version.

\smallskip
The computational domain is set as $[0, 1]^2$. The source term $f$ is set as zero.
The boundary condition is set as $g_D = 1$ on the left boundary, $g_D = 0$ on the right boundary, and $g_N = 0$ on the top and bottom boundaries.
The conductivity field $\boldsymbol{K}$ is set as follows:
\begin{equation}
	\label{cond:training}
	\boldsymbol{K}(x,y) = \begin{bmatrix}
		k(x,y) & 0 \\
		0 & k(x,y)
	\end{bmatrix}, \quad
	k(x,y) = \begin{cases}
		10^{-6}, & \text{if } 0.4 \leq x, y \leq 0.6, \\
		1, & \text{otherwise}.
	\end{cases}
\end{equation}
\begin{figure}[htbp]
\centering
\subfloat[Training conductivity field]{\includegraphics[width=0.45\textwidth]{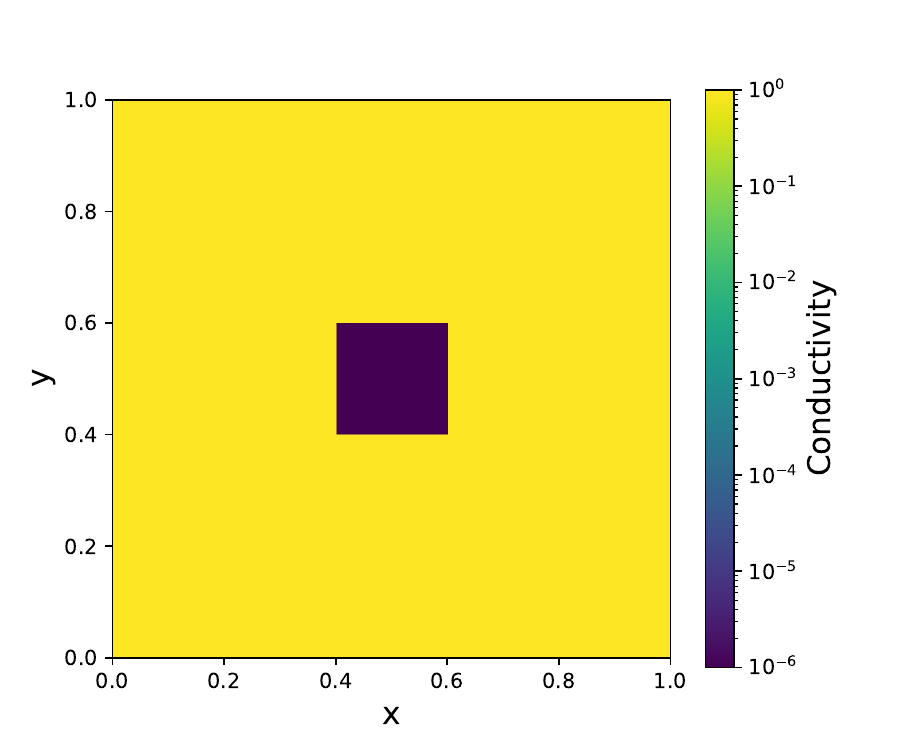}\label{fig:darcy-conductivity}}
\hfill
\subfloat[Training loss evolution]{\includegraphics[width=0.45\textwidth]{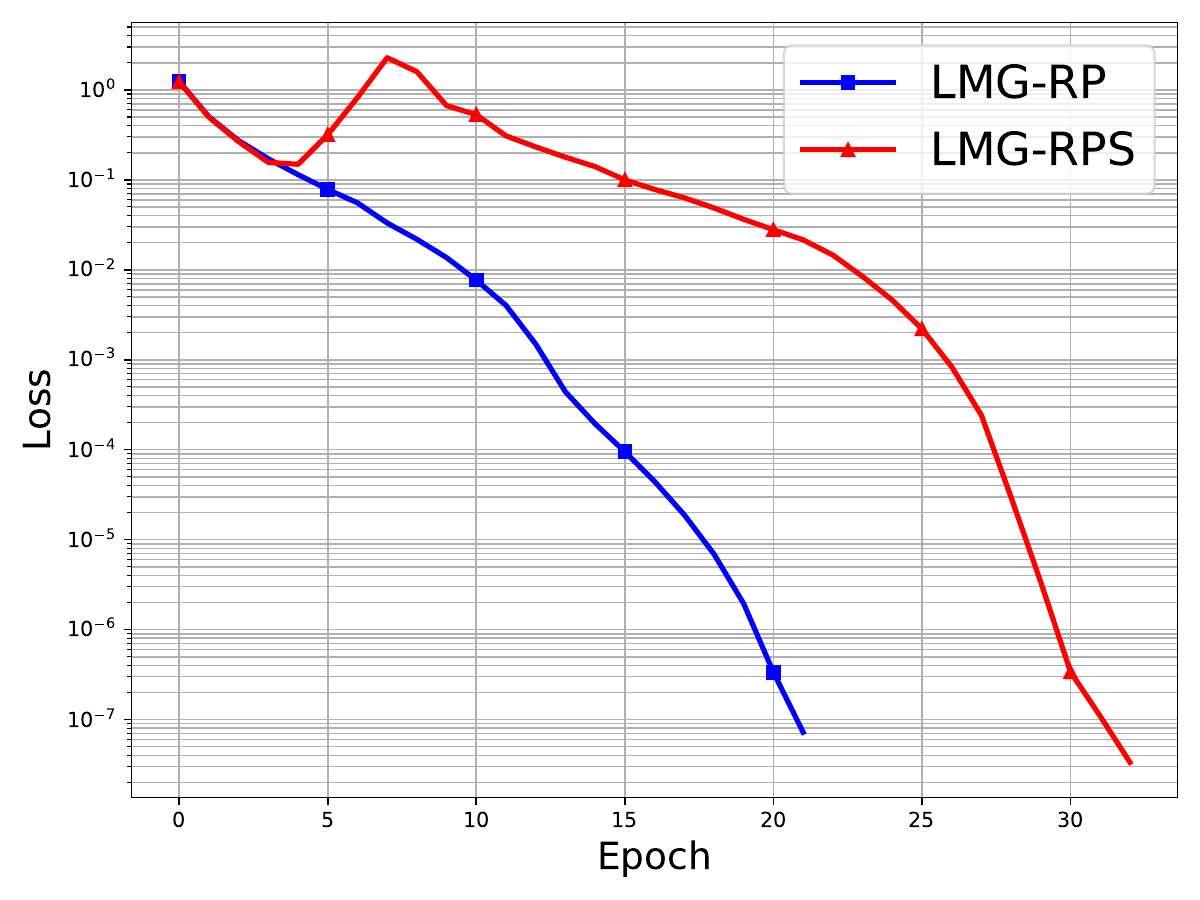}\label{fig:darcy-loss}}
\caption{The training conductivity field $\boldsymbol{K}$ and the loss evolution during training for the Darcy equation \cref{eq:darcy}.}
\label{fig:darcy-training}
\end{figure}
The conductivity field is visualized in \cref{fig:darcy-conductivity}.

\smallskip
The loss evolution during the training process for both LMG-RP and LMG-RPS is shown in \cref{fig:darcy-loss}.
LMG-RP converges in about 21 epochs and is more stable because it trains only the restriction and prolongation graph convolutions.
LMG-RPS has more flexibility but also more parameters, leading to a longer training process and a mild initial loss fluctuation.
Both variants significantly improve the piecewise constant multigrid baseline.

\begin{remark}
The purpose of this experiment is to demonstrate how learnable components can improve the convergence when the multigrid method is not optimal for a given problem.
The classical multigrid with bilinear interpolation already performs very well:
for the training conductivity field \cref{cond:training}, it converges in only about 10 iterations.
This rapid convergence means that there are not enough training epochs for a learnable model built on top of the bilinear interpolation to discover meaningfully better weights,
since the solver reaches the tolerance before the parameters can be sufficiently optimized.
Consequently, the learned model based on bilinear interpolation exhibits poor generalization ability.
In contrast, the piecewise constant interpolation converges slowly and thus provides enough training epochs for the graph convolutions to learn effective corrections.
As shown in the subsequent experiments, both LMG-RP and LMG-RPS built on piecewise constant interpolation can outperform the classical multigrid solver with bilinear interpolation.
\end{remark}

\smallskip
We then compare LMG-RP, LMG-RPS, and classical multigrid with bilinear interpolation on three conductivity fields shown in \cref{fig:darcy-test-conductivity}; a random initial guess is used in all tests.
\begin{figure}
\centering
\subfloat[Testing conductivity field I]{\includegraphics[width=0.31\textwidth]{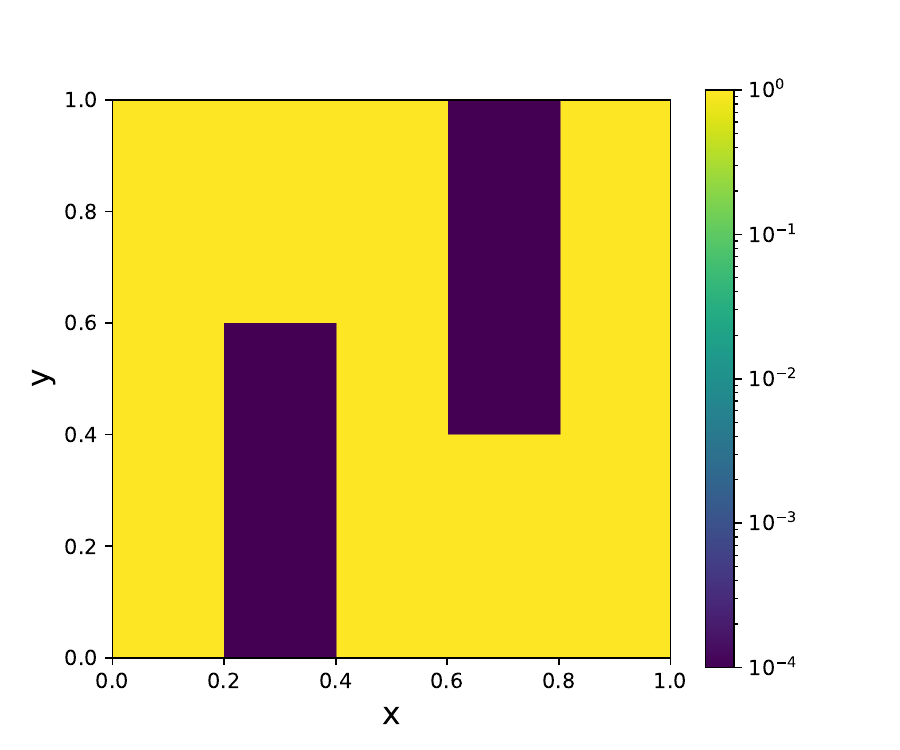}}
\hfill
\subfloat[Testing conductivity field II]{\includegraphics[width=0.31\textwidth]{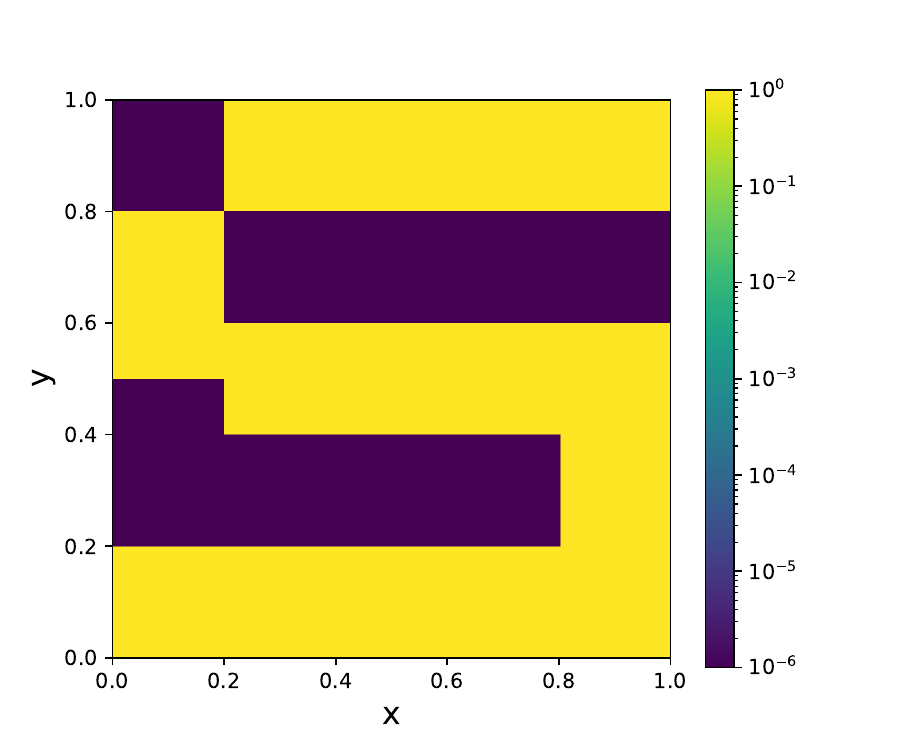}}
\hfill
\subfloat[Testing conductivity field III]{\includegraphics[width=0.31\textwidth]{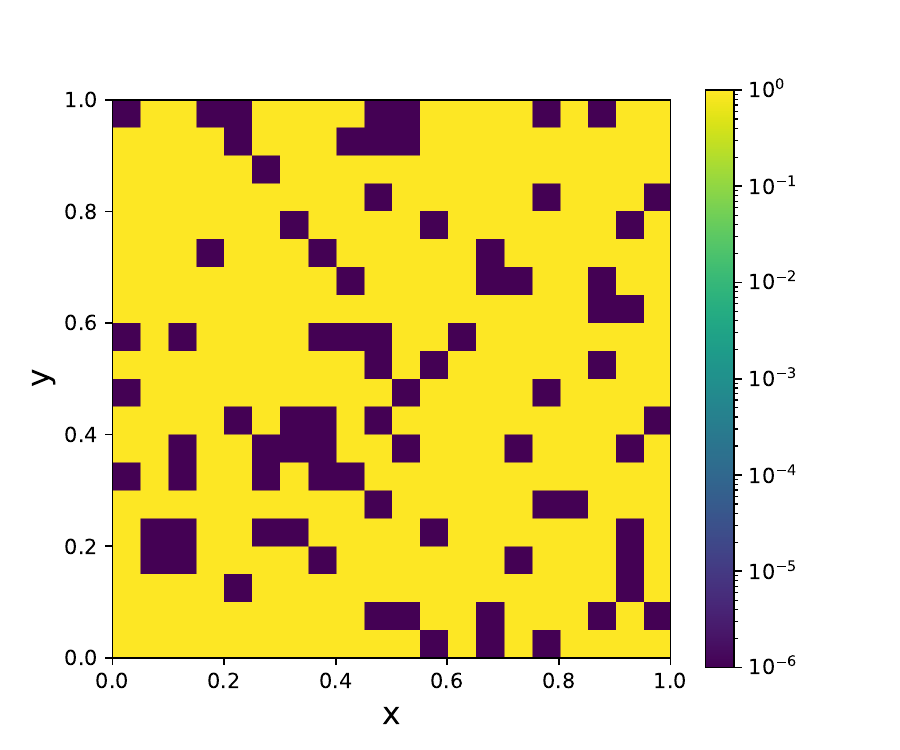}}
\\
\subfloat[Solution - field I]{\includegraphics[width=0.31\textwidth]{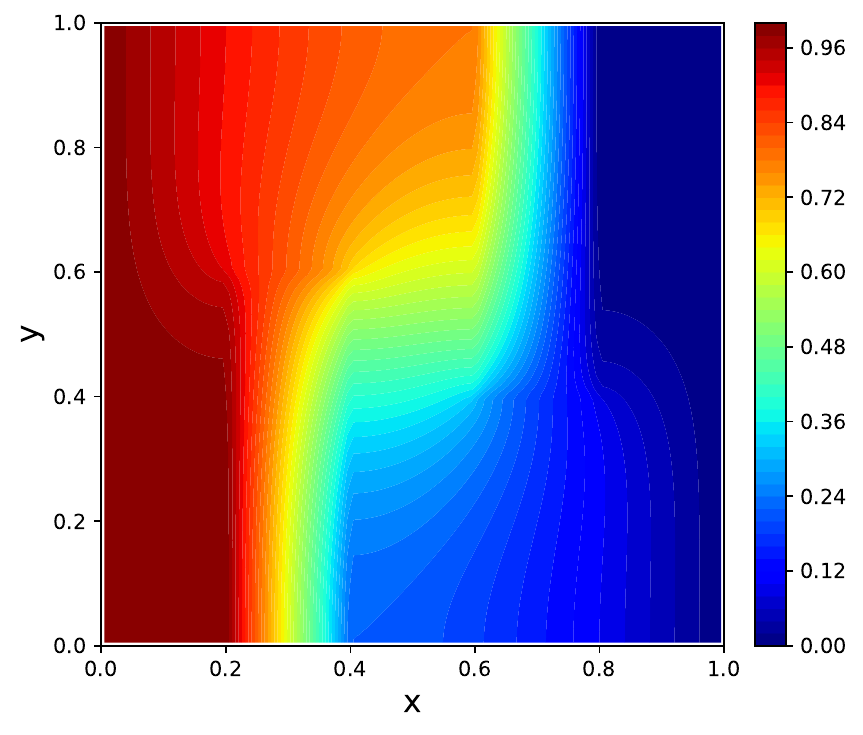}}
\hfill
\subfloat[Solution - field II]{\includegraphics[width=0.31\textwidth]{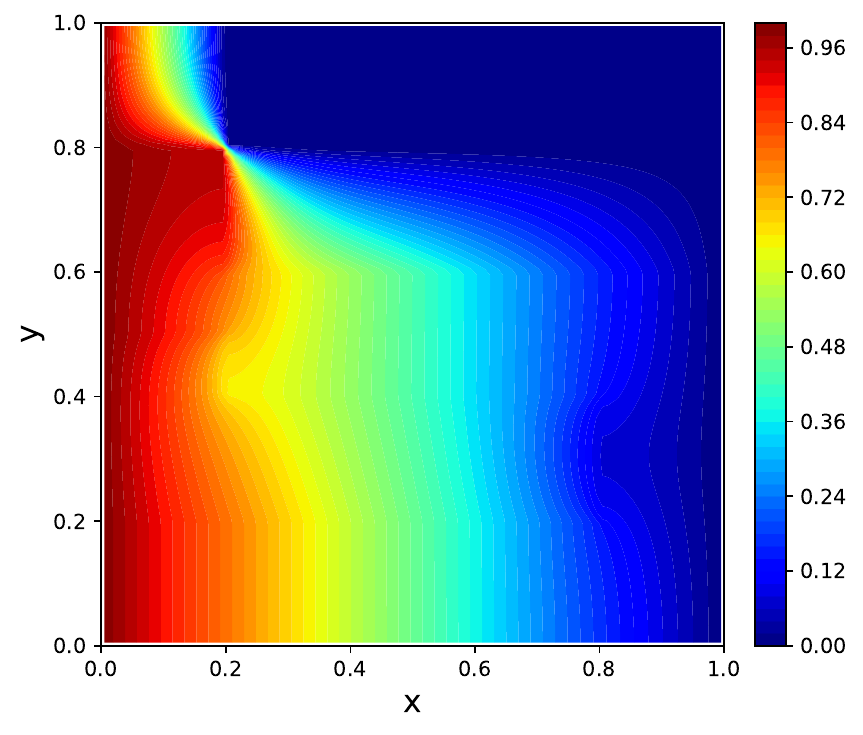}}
\hfill
\subfloat[Solution - field III]{\includegraphics[width=0.31\textwidth]{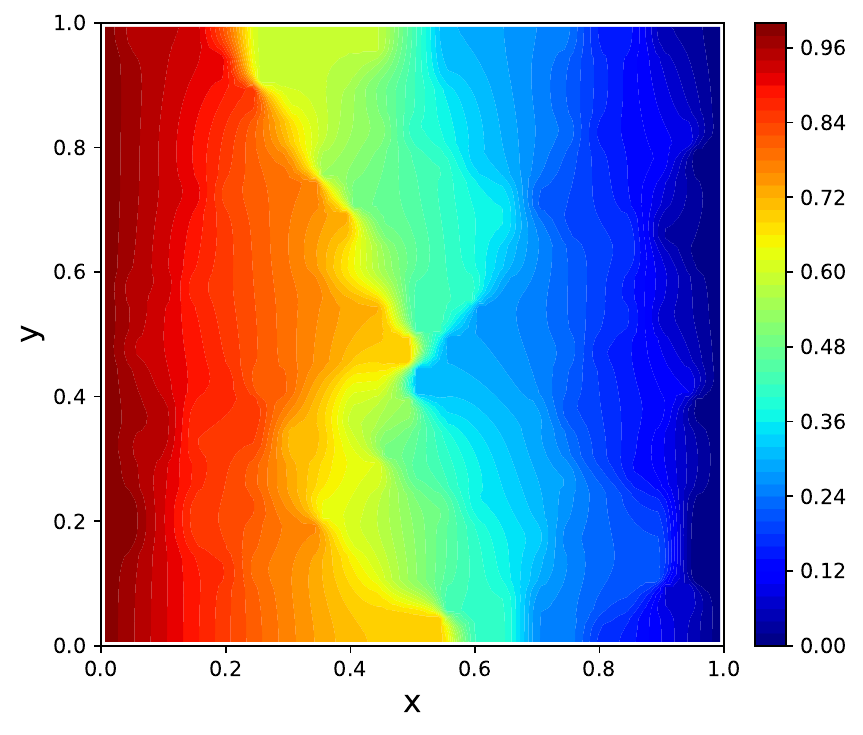}}
\caption{The conductivity fields for testing the generalizability of the learnable multigrid solver for the Darcy equation \cref{eq:darcy} and the corresponding numerical solutions.}
\label{fig:darcy-test-conductivity}
\end{figure}
Field I contains two low-conductivity regions with contrast $10^4$.
Field II has two high-conductivity channels with contrast $10^6$, producing a sharp interface in the solution.
Field III is the Durlofsky field \cite{durlofsky1994accuracy}, interpolated from $40\times40$ to $80\times80$ for testing the generalization ability to different mesh sizes.
The boundary conditions and source term are the same as in training.

\subsubsection{Performance and discussion}
We summarize the results in \cref{tab:darcy-results}.

\begin{table}[htbp]
\centering
\small
\caption{Iteration counts (Iter) and wall clock time (Time, in seconds) for the Darcy equation \cref{eq:darcy} with different conductivity fields.}
\label{tab:darcy-results}
\begin{tabular}{c cc c@{\hspace{3pt}}c c@{\hspace{3pt}}c}
\toprule
 & \multicolumn{2}{c}{Classical MG} & \multicolumn{2}{c}{LMG-RPS} & \multicolumn{2}{c}{LMG-RP} \\
\cmidrule(lr){2-3} \cmidrule(lr){4-5} \cmidrule(lr){6-7}
Field & Iter & Time & Iter & Time & Iter & Time \\ 
\midrule
I   & 14 & 0.18 & 9  & 0.12 & 13 & 0.12 \\
II  & 38 & 0.45 & 17 & 0.24 & 13 & 0.12 \\
III & 30 & 0.25 & 18 & 0.15 & 15 & 0.09 \\
\bottomrule
\end{tabular}
\end{table}

\begin{figure}[htbp]
\centering
\subfloat[Field I]{\includegraphics[width=0.31\textwidth]{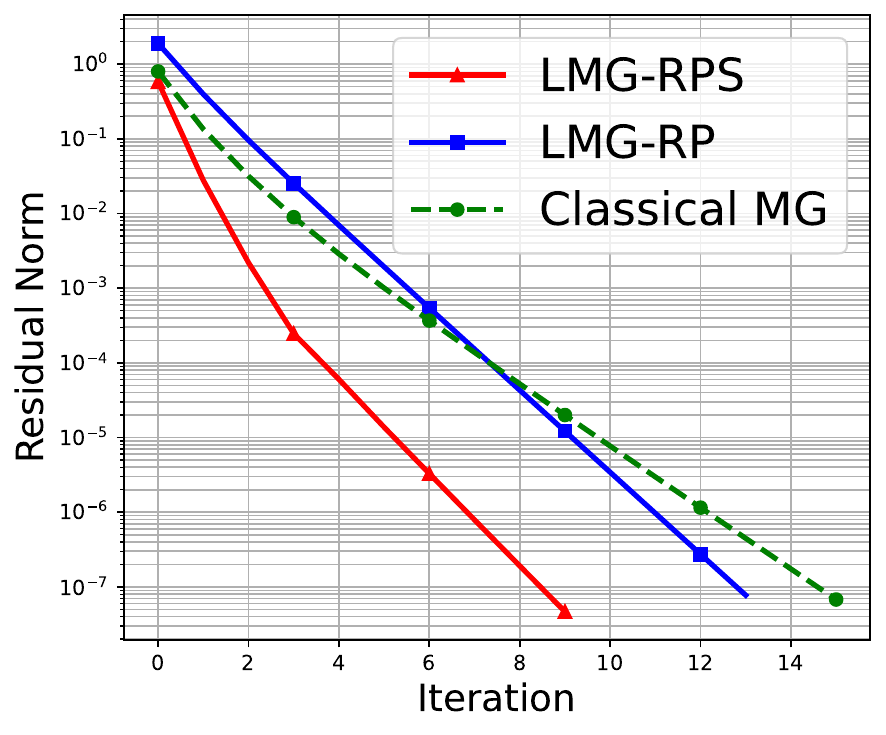}}
\subfloat[Field II]{\includegraphics[width=0.31\textwidth]{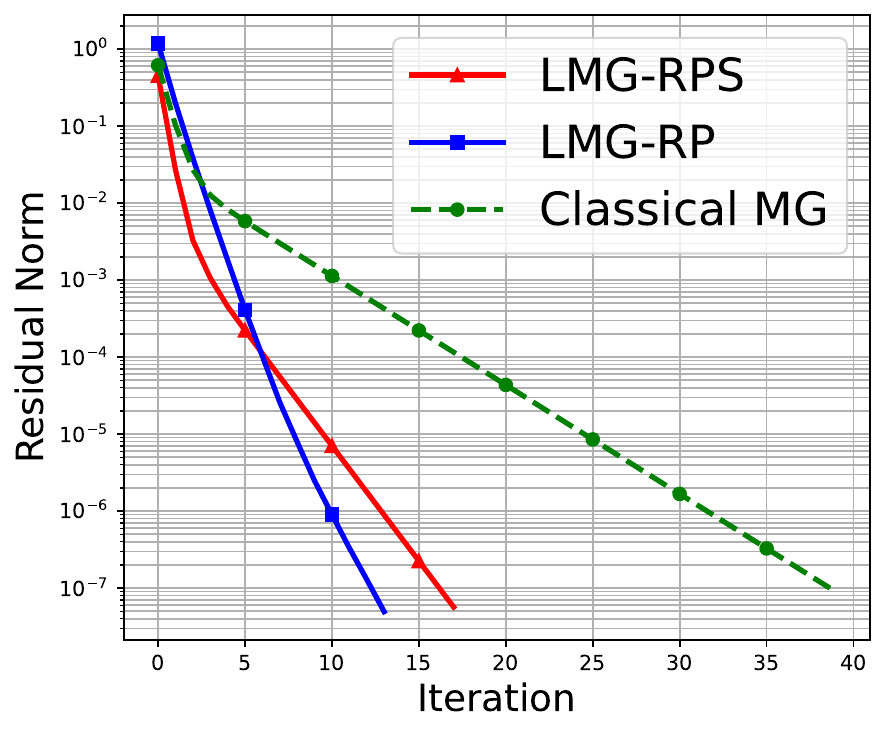}}
\subfloat[Field III]{\includegraphics[width=0.31\textwidth]{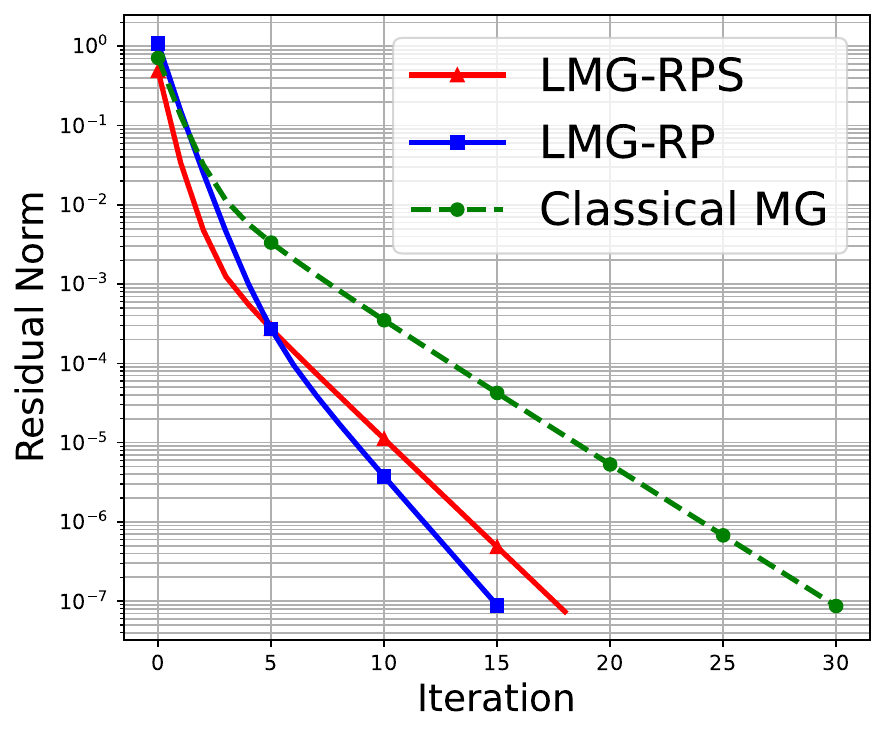}}
\caption{Residual evolution profile for the Darcy equation \cref{eq:darcy} with the three testing conductivity fields.}
\label{fig:darcy-residual-evolution}
\end{figure}
Both learnable solvers outperform classical multigrid for all testing fields, indicating generalization beyond the training coefficient.
Their V-cycles are also cheaper because the piecewise constant interpolation leads to sparser coarse operators than bilinear interpolation.
LMG-RPS is best for Field I, while LMG-RP is faster for Fields II--III and has nearly constant iteration counts across all tests.
Since the Jacobi smoother is already effective here, LMG-RP is the more practical choice.

\smallskip
We also plot the frequency evolution of the residual for the classical multigrid solver and the trained learnable multigrid solver in \cref{fig:darcy-error-frequency}.
\begin{figure}
\centering
\subfloat[Classical MG - Field I]{\includegraphics[width=0.31\textwidth]{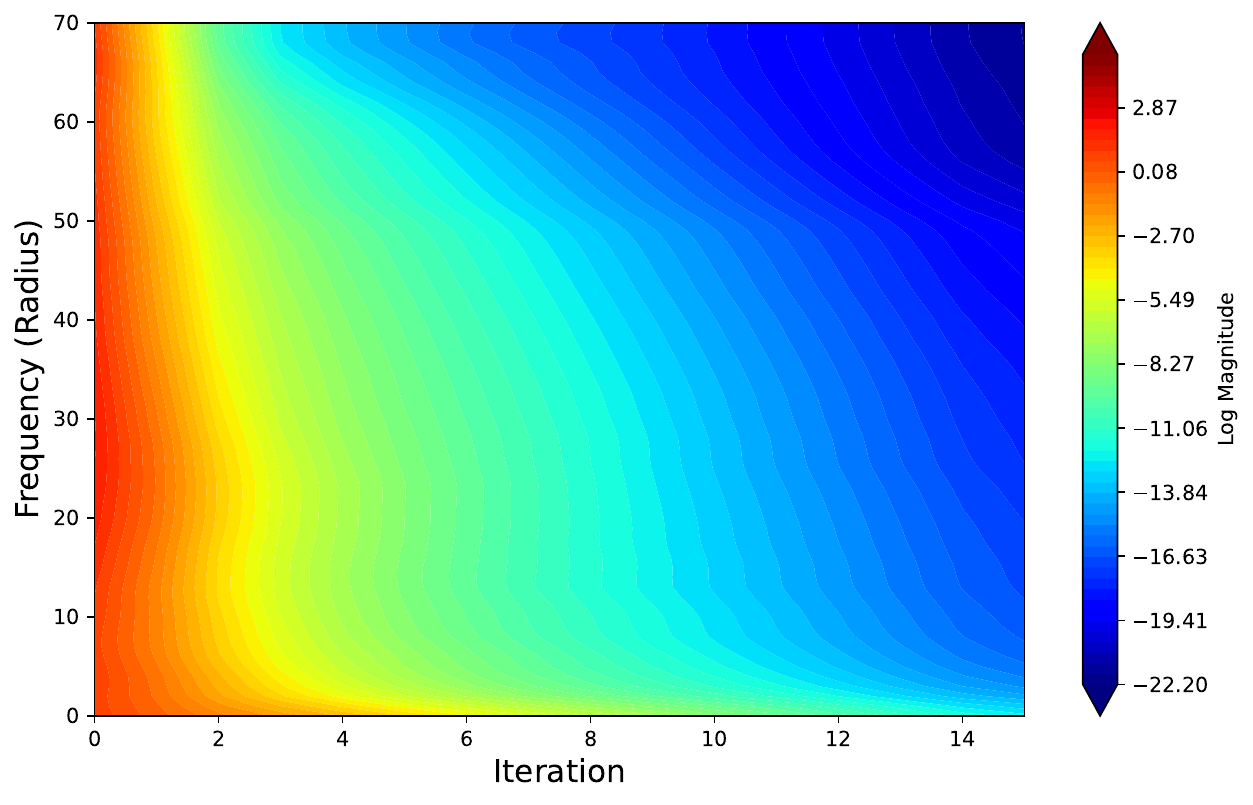}}
\hfill
\subfloat[Classical MG - Field II]{\includegraphics[width=0.31\textwidth]{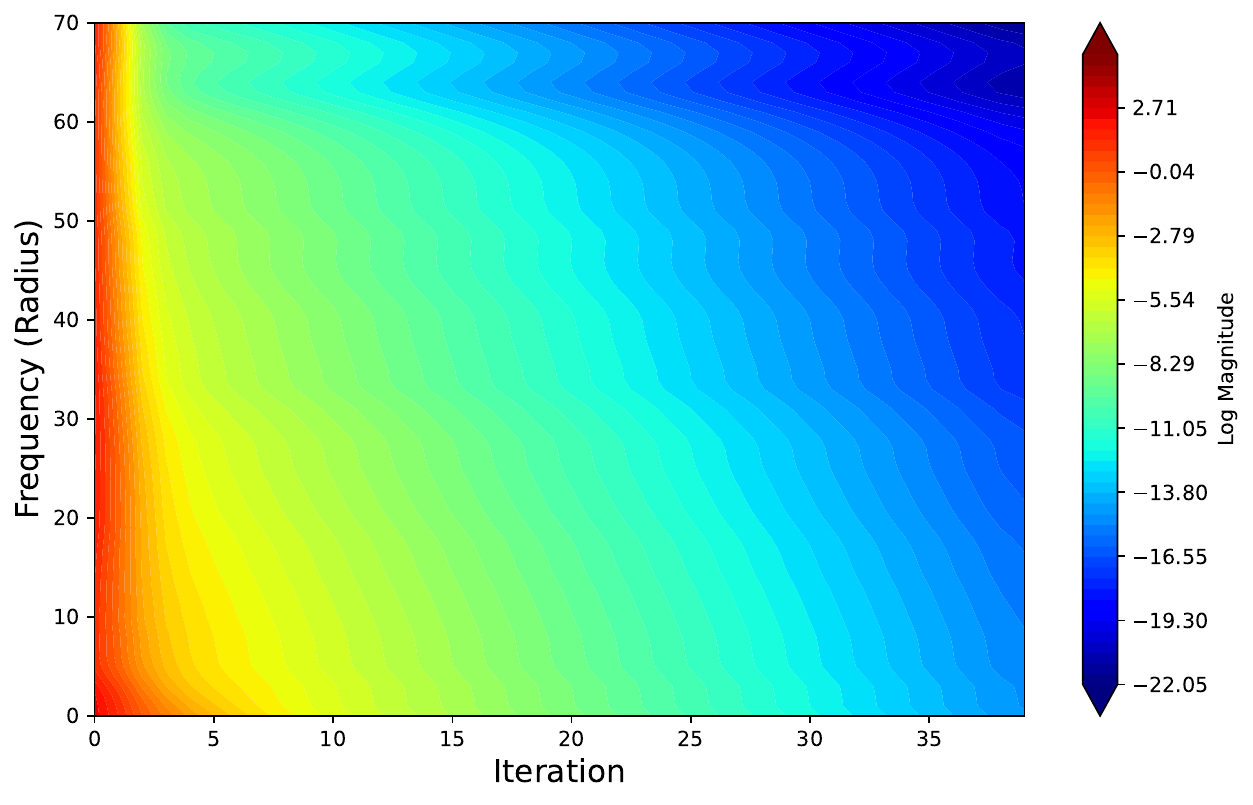}}
\hfill
\subfloat[Classical MG - Field III]{\includegraphics[width=0.31\textwidth]{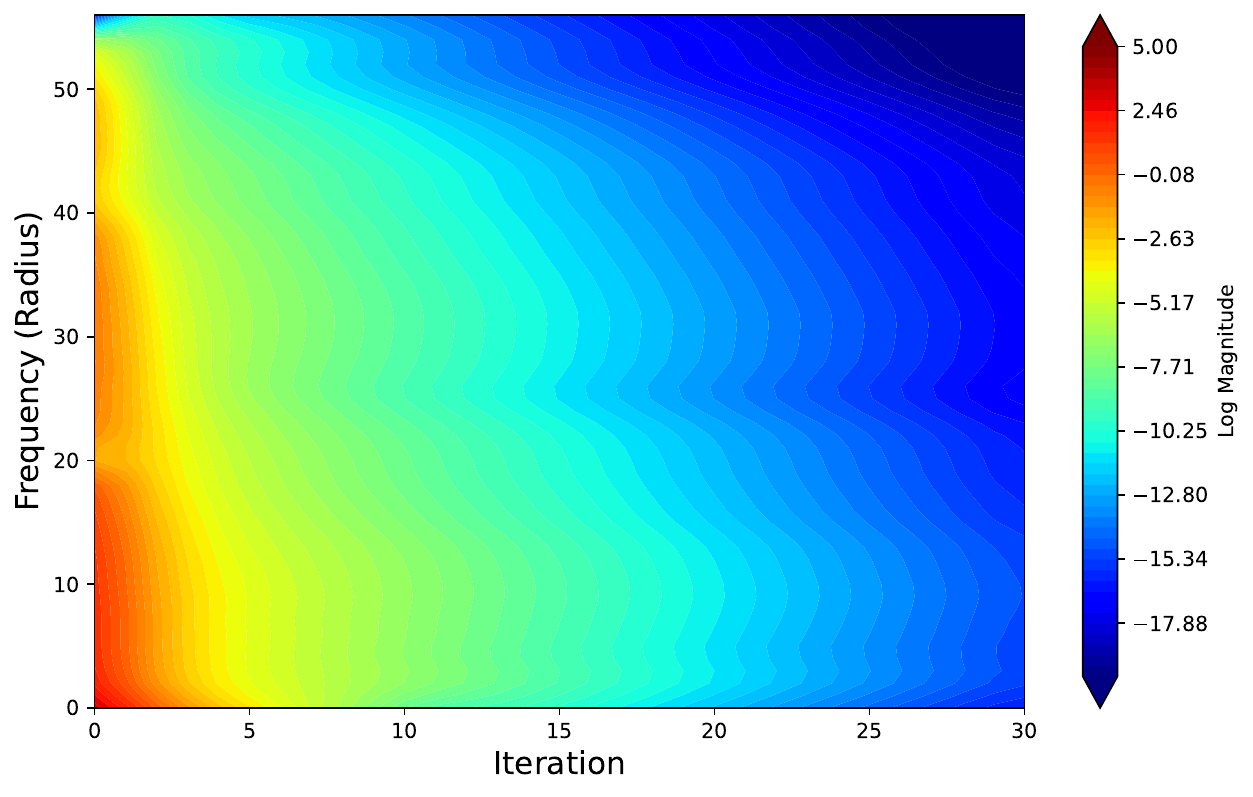}}
\\
\subfloat[LMG-RPS - Field I]{\includegraphics[width=0.31\textwidth]{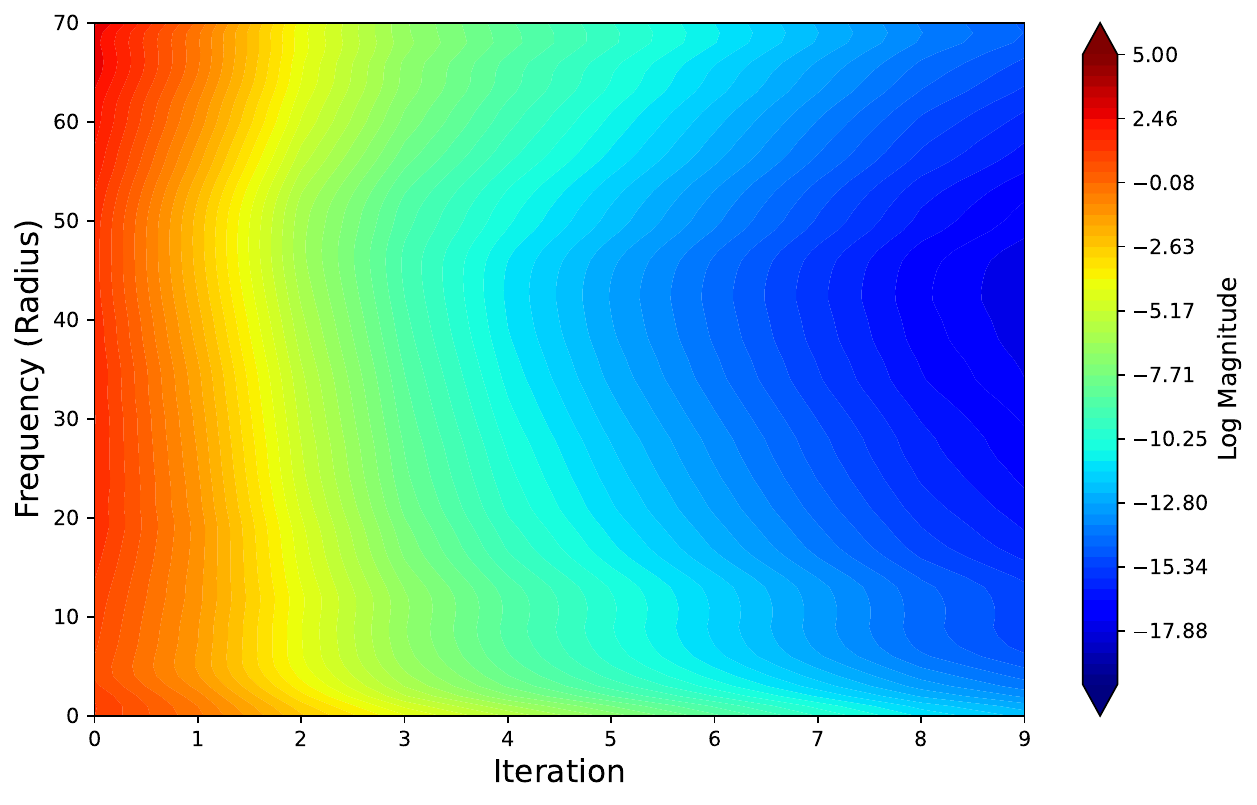}}
\hfill
\subfloat[LMG-RPS - Field II]{\includegraphics[width=0.31\textwidth]{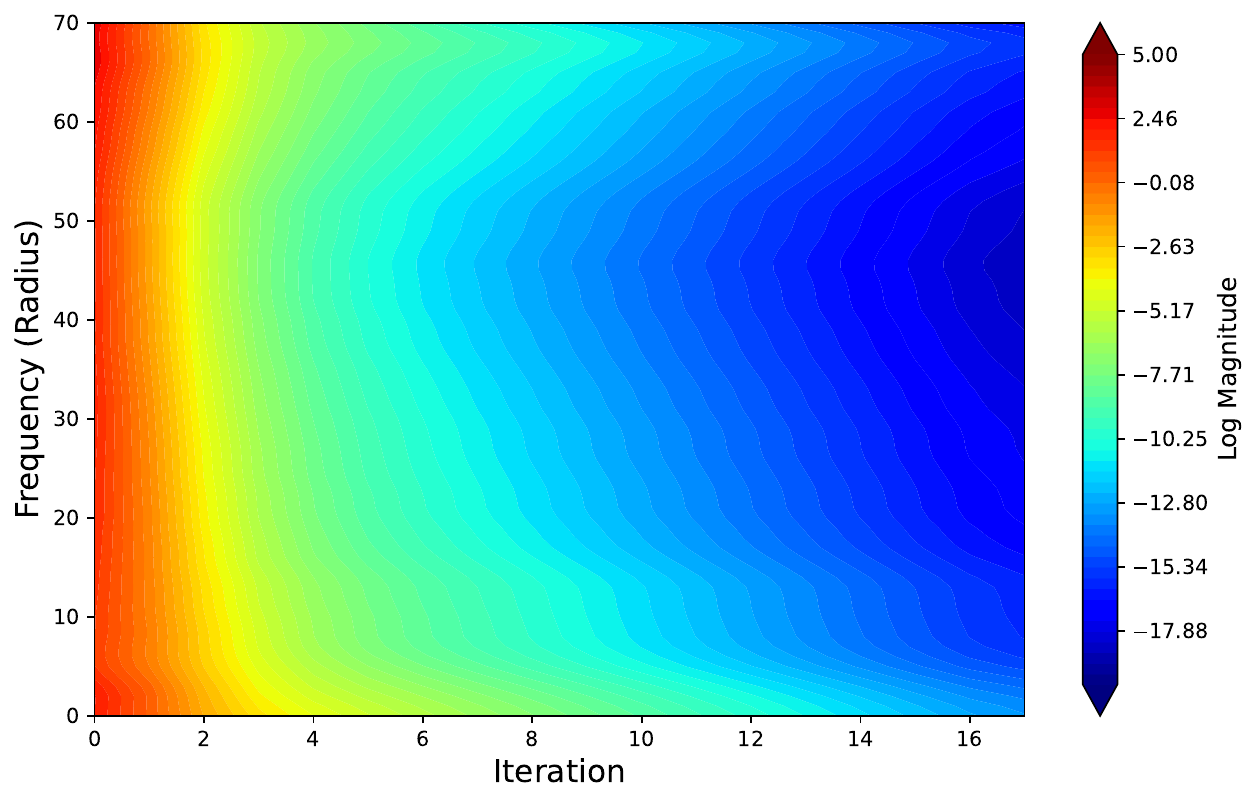}}
\hfill
\subfloat[LMG-RPS - Field III]{\includegraphics[width=0.31\textwidth]{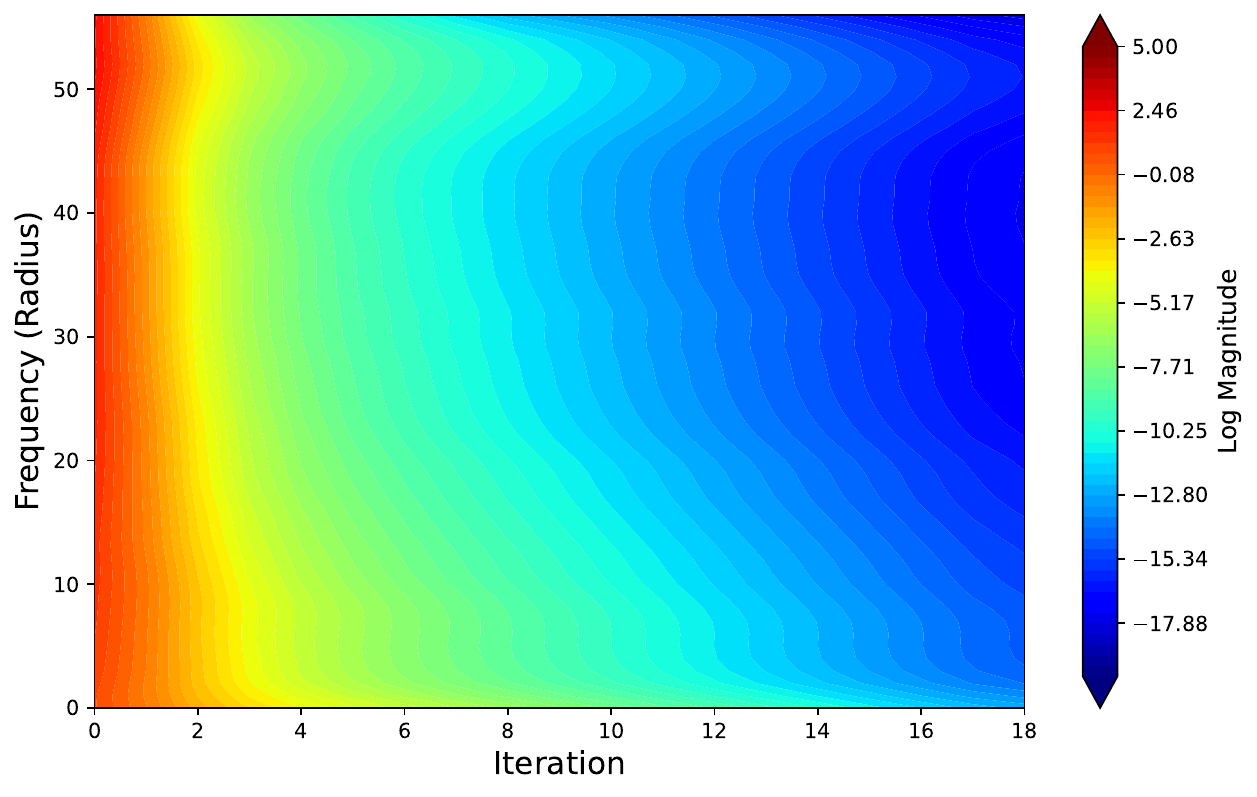}}
\\
\subfloat[LMG-RP - Field I]{\includegraphics[width=0.31\textwidth]{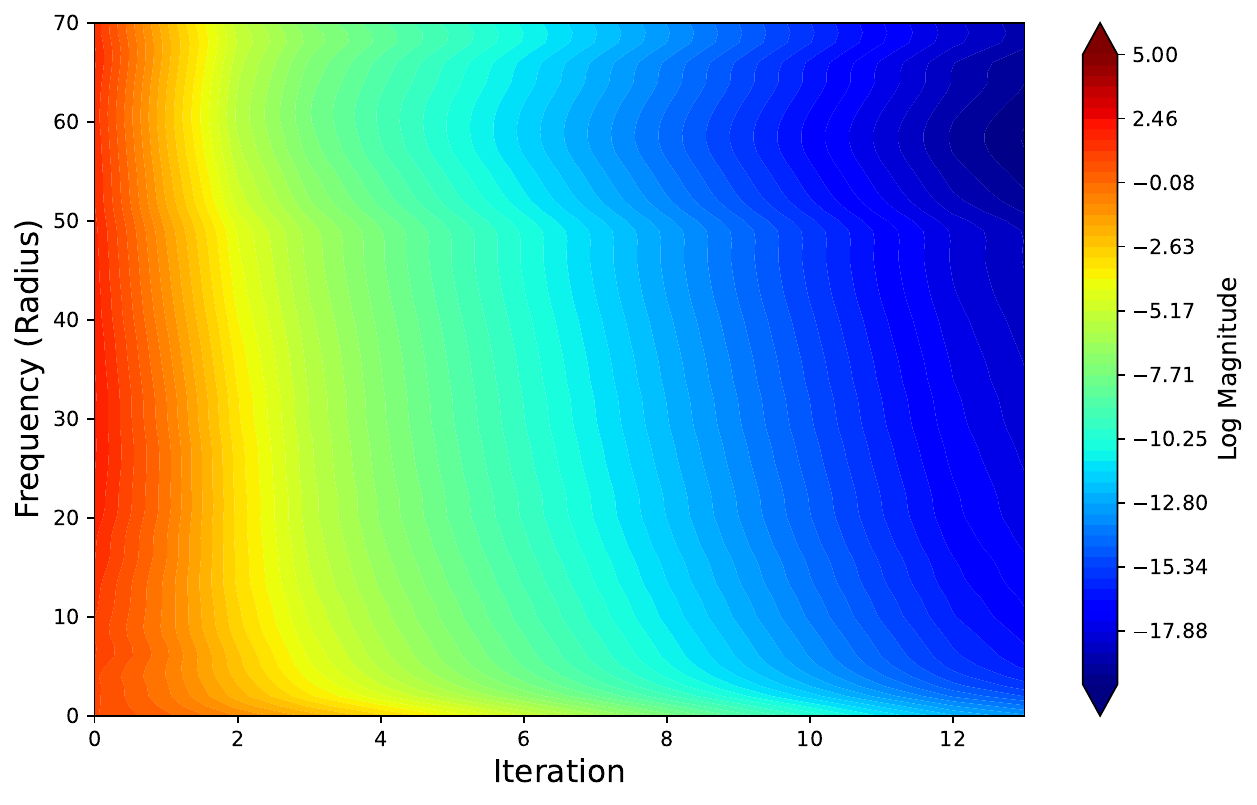}}
\hfill
\subfloat[LMG-RP - Field II]{\includegraphics[width=0.31\textwidth]{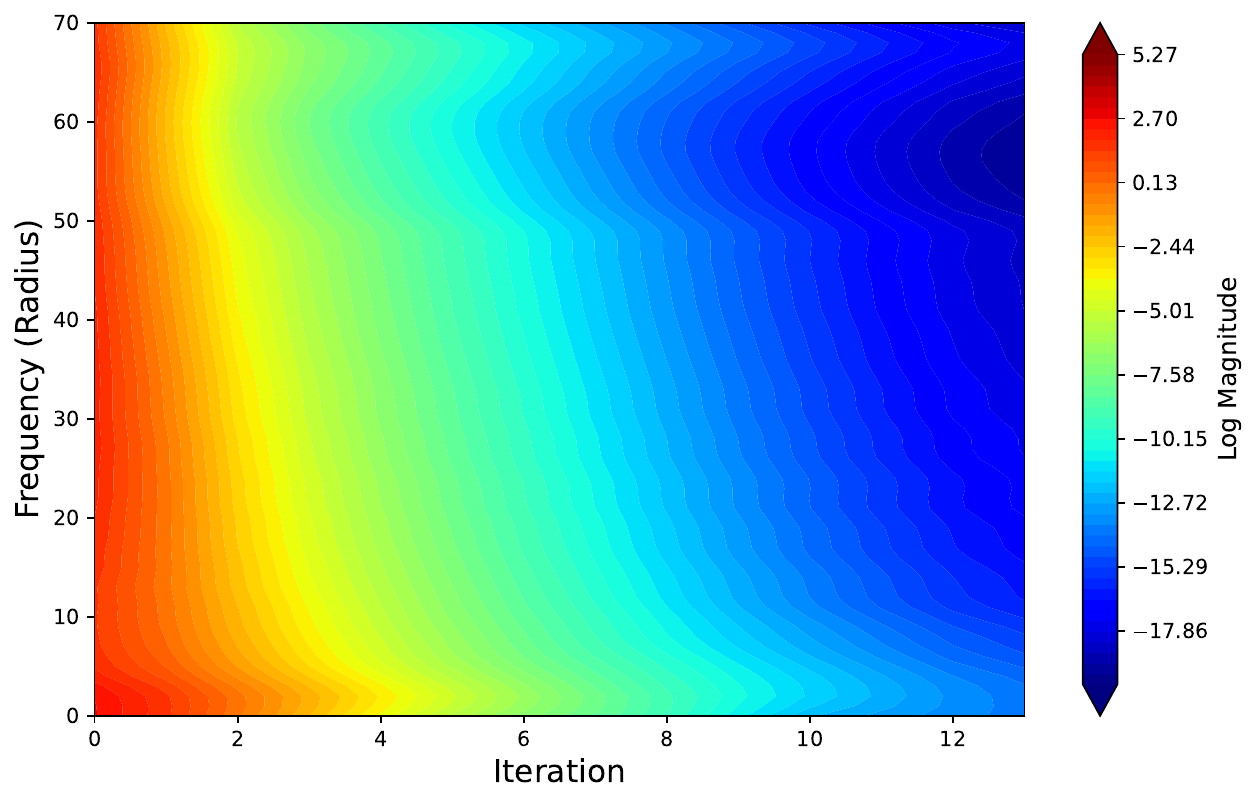}}
\hfill
\subfloat[LMG-RP - Field III]{\includegraphics[width=0.31\textwidth]{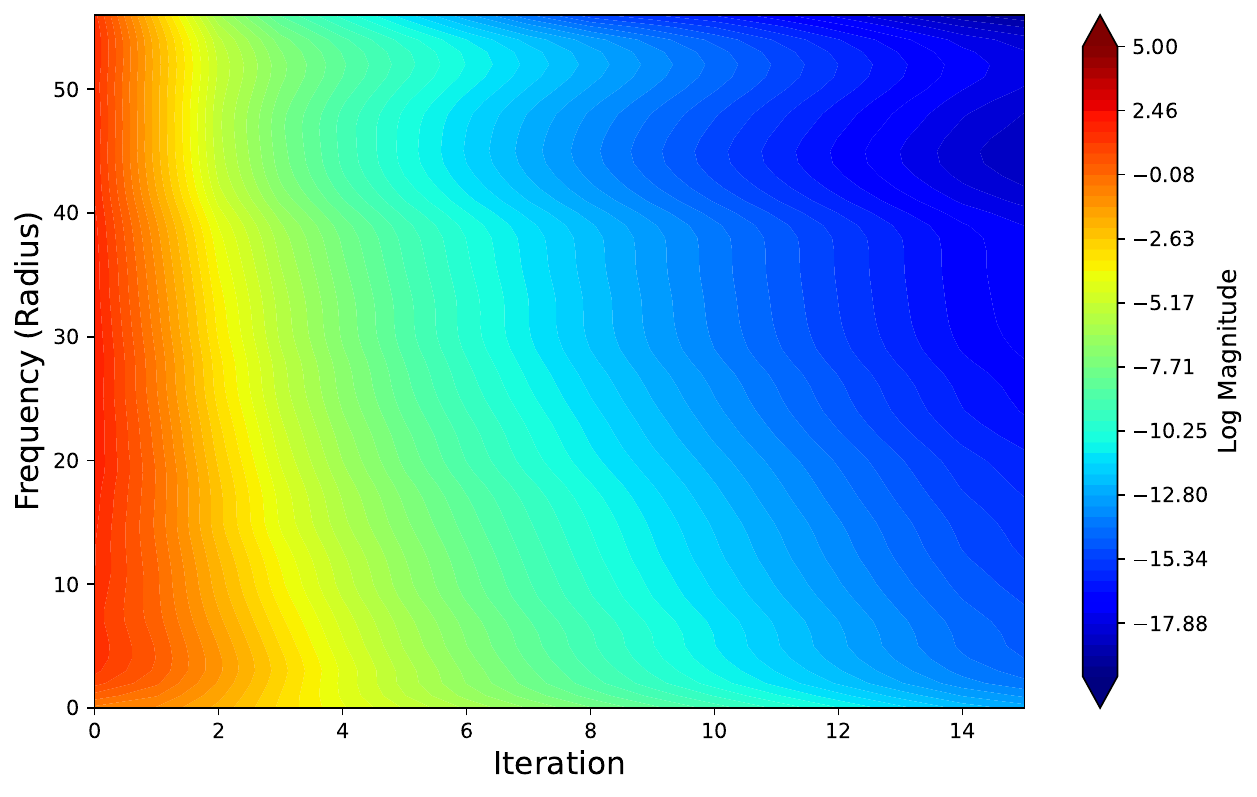}}
\caption{The frequency evolution of the residual for the classical multigrid solver (top row), the LMG-RPS solver (mid row), and the LMG-RP solver (bottom row) for the Darcy equation \cref{eq:darcy} with the three testing conductivity fields.}
\label{fig:darcy-error-frequency}
\end{figure}
For the residual $r^m$ at iteration $m$, we reshape it as a 2D array and apply the two-dimensional fast Fourier transform (FFT):
\begin{equation}
\widehat r^m_{pq}=\left|\mathrm{FFT}(r^m)_{pq}\right|, \quad
E_m(\rho)=\frac{1}{|\mathcal{S}_\rho|}\sum_{(p,q)\in\mathcal{S}_\rho}\widehat r^m_{pq},
\end{equation}
where $\mathcal{S}_\rho=\{(p,q):\operatorname{round}(\sqrt{p^2+q^2})=\rho\}$.
The y-axis is the frequency radius $\rho$, and the color shows $\log E_m(\rho)$.

\smallskip
The contours show the expected behavior of classical multigrid: high-frequency errors decay first, while some low-frequency components persist.
LMG-RPS and LMG-RP reduce both ranges more uniformly, especially near frequency radius 40--50, which is consistent with the frequency-dependent damping described in \cref{subsec:spectral-analysis}.

\subsubsection{Sensitivity to the initialization}
We also investigate the sensitivity of the training process to the initialization of the learnable parameters.
We tested a fully learnable model $\{\Theta_R,\Theta_P,\Theta_{B,\ell},\eta_\ell,\theta_\ell\}$ with randomly initialized graph convolutions, very small scaling factors $\eta_\ell=\theta_\ell=10^{-8}$, and learning rate $10^{-4}$.
The loss evolution is shown in \cref{fig:darcy-loss-random-init}.
\begin{figure}[htbp]
\centering
\includegraphics[width=0.48\textwidth]{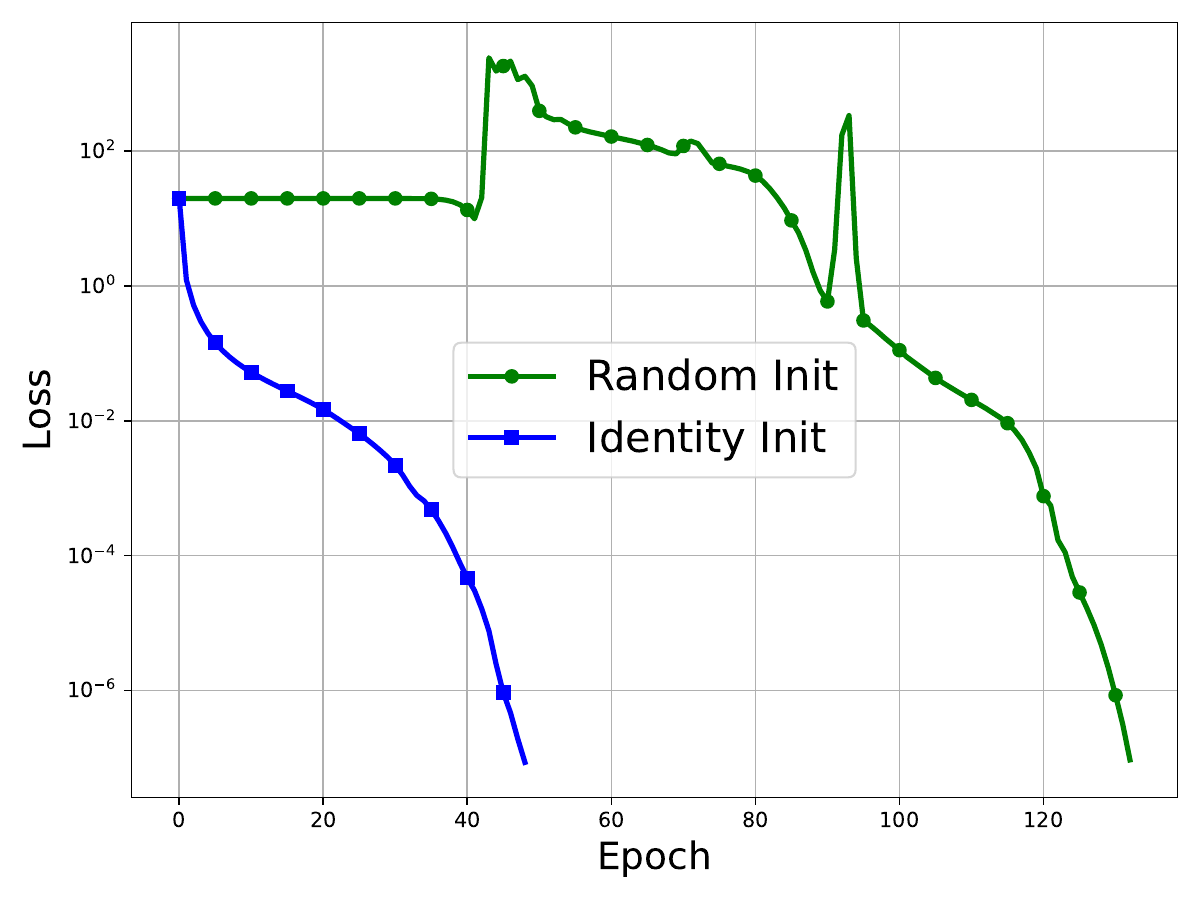}
\caption{The loss evolution during the training process for the Darcy equation \cref{eq:darcy} with random initialization and identity initialization of the learnable parameters with the same learning rate.}
\label{fig:darcy-loss-random-init}
\end{figure}
Although one run converged after about 130 epochs, its training process was less stable than the identity-initialized model.
The generalization ability of the random-initialized model is also limited, with iteration counts of 7 and 12 for Fields I and III, respectively.
For Field II, the random-initialized model diverged.
We notice that the random initialization has less iterations for Field I and III than the identity initialization.
This is possibly due to the larger parameter space explored by the random initialization.

\smallskip
However, this approach is less stable and can lead to divergence, as observed in Field II.
Several other random initializations diverged, and in a further test only Field II converged within more than 200 iterations.
This confirms the advantage of the identity initialization, which starts from a stable classical multigrid iteration.

\subsection{Helmholtz equation}\label{subsec:helmholtz}
Then, we consider the Helmholtz equation with the impedance boundary condition as our second model problem:
\begin{equation}\label{eq:helmholtz}
	\begin{aligned}
		&-\Delta u - k^2 u = f, \quad \text{in } \Omega, \\
		&\frac{\partial u}{\partial n} - ik u = 0, \quad \text{on } \partial \Omega,
	\end{aligned}
\end{equation}
where $k$ is the wave number, and $i$ is the imaginary unit.

\smallskip
In general, the wave number $k$ and the mesh size $h$ should satisfy the condition $k h \leq C$ to reduce the pollution effect.
We discretize the domain with nonuniform triangular meshes generated by pygmsh \cite{geuzaine_gmsh_2009} and use geometric coarsening.
The prolongation operators $P_\ell$ are built by barycentric interpolation, and $R_\ell=P_\ell^T$.
We follow the setting in \cite{elman_multigrid_2001} and 
set the smoother as the Jacobi method with one step of pre-smoothing and post-smoothing when $kh_\ell \leq 0.5$ 
and the GMRES smoother with one step of pre-smoothing and two steps of post-smoothing when $kh_\ell > 0.5$.
We fix the number of inner iterations for the GMRES smoother as 4.
The classical multigrid for comparison in this experiment uses the same settings.

\smallskip
The computational domain is set as a unit square with its center located at the origin.
We use smoothed disk source terms for training and testing.
Let $\boldsymbol{x}=(x,y)^T$, $r_s=0.02$, $\varepsilon_s=0.001$, $\boldsymbol{c}_0=(0,0)^T$, and $\boldsymbol{c}_{\pm}=(0,\pm 0.25)^T$.
The training source has one center at the origin, while the testing source has two centers:
\begin{align}
f_{\mathrm{train}}(\boldsymbol{x})
&=100\left(1-\tanh\left(\frac{\|\boldsymbol{x}-\boldsymbol{c}_0\|_2^2-r_s^2}{\varepsilon_s}\right)\right), \label{eq:helmholtz-training-source}\\
f_{\mathrm{test}}(\boldsymbol{x})
&=100\left(1-\prod_{\boldsymbol{c}\in\{\boldsymbol{c}_{+},\boldsymbol{c}_{-}\}}
\tanh\left(\frac{\|\boldsymbol{x}-\boldsymbol{c}\|_2^2-r_s^2}{\varepsilon_s}\right)\right). \label{eq:helmholtz-testing-source}
\end{align}

\smallskip
To accommodate complex-valued solutions inherent to the problem, the input and output feature dimensions are set to 2 in our numerical experiments.
The weight matrices $W_1$ and $W_2$ are initialized as:
\begin{equation*}
	W_1 = \begin{bmatrix} 1 & 0 \\ 0 & 1 \end{bmatrix}, \quad
	W_2 = \begin{bmatrix} 0 & 0 \\ 0 & 0 \end{bmatrix}.
\end{equation*}
The off-diagonal entries are trainable, so the graph convolutions can mix real and imaginary parts.
All smoother graph convolutions $\Phi_{B,\ell}$ are trained, while $\Phi_R$ and $\Phi_P$ are shared across levels.

\smallskip
We first present the training results of the proposed learnable multigrid solver as an iterative solver. 
The wave number is chosen as $16\pi$.
The mesh size is fixed as $h_{\min} = 1/256$ for all the experiments in this subsection.
Thus $kh_{\min}\approx 0.196$, similar to the setting in \cite{elman_multigrid_2001}.

\smallskip
We vary the number of levels $J$ for comparisons.
For Helmholtz problems, too many levels can hurt the classical method because coarse-grid correction becomes less effective, while in the learnable method larger $J$ also increases the number of trainable parameters.

\smallskip
The training and testing results as iterative solvers are summarized in \cref{tab:helmholtz-results-iterative}. The wall clock time for training is also reported in the table.

\begin{table}[htbp]
\centering
\small
\caption{Training epochs and inference performance as an iterative solver (V-cycle) with wall clock time (in seconds) for the Helmholtz equation \cref{eq:helmholtz} with different number of levels $J$.}
\label{tab:helmholtz-results-iterative}
\begin{tabular}{ccccc}
\toprule
	\multirow{2}{*}{$J$} & \multicolumn{2}{c}{Classical MG} & \multicolumn{2}{c}{LMG} \\
	\cmidrule(lr){2-3} \cmidrule(lr){4-5}
	& Steps & Time & \multicolumn{1}{c@{\hspace{20pt}}}{\hspace{20pt}Steps} & Time \\
\midrule
	\multicolumn{5}{l}{\textit{LMG strategy: partial weight transfer}} \\
	\midrule
	5 & 19 & 1.97 & 14 & 1.67 \\
	6 & 20 & 2.12 & 14 & 1.71 \\
	7 & 20 & 2.18 & 14 & 1.77 \\
	8 & 20 & 2.29 & 14 & 1.86 \\
\bottomrule
\end{tabular}
\end{table}

Training converges in about 40 epochs for all $J$.
Classical multigrid is effective for $J=5$ but degrades for deeper hierarchies, whereas the learnable solver remains robust and improves slightly as $J$ increases.

\smallskip
The inference time of one V-cycle of the learnable multigrid solver is around 1.5 times that of the classical multigrid solver due to the additional graph convolution operations.
The trained model can decrease the iteration counts, but the total computational time is not reduced due to the increased cost per iteration.

\smallskip
Besides using the trained model as an iterative solver,
we also utilize the trained model as preconditioners for Krylov subspace methods for comparisons.
We combine both the learnable and classical multigrid solvers as preconditioners for the FGMRES \cite{saad1993flexible} method.

\smallskip
In this experiment, the trained model is also applied to construct the full multigrid (FMG) as preconditioners for Krylov subspace methods.
The FMG method, also known as nested iteration, starts from the coarsest grid and progressively refines the solution on finer grids. At each level, the solution from the coarser grid is interpolated to provide a good initial guess, which is then improved by one or more V-cycles. 
The parameter number of V-cycles per level is set to 1 in our experiments.
Note that the restriction and prolongation operators in the FMG method are the same as those in the V-cycle.

The maximum number of iterations for the Krylov subspace methods is set to 40.
The results are summarized in \cref{tab:helmholtz-results-precond}. 

\begin{table}[htbp]
\centering
\small
\caption{Comparison of iteration counts and wall clock time (in seconds) when used as preconditioners for FGMRES at inference. 
The -- notation indicates that the solver did not converge within the maximum number of iterations.}
\label{tab:helmholtz-results-precond}
\begin{tabular}{ccccc}
\toprule
\multirow{2}{*}{$J$} & \multicolumn{2}{c}{V-cycle preconditioner} & \multicolumn{2}{c}{FMG preconditioner} \\ 
\cmidrule(lr){2-3} \cmidrule(lr){4-5}
 & Classical MG& Learnable MG & Classical MG & Learnable MG \\ 
\midrule
5 & 37 (1.51) & 25 (1.41) & 11 (0.85) & 9 (0.87) \\
6 & -  & 22 (1.29) & 19 (1.54) & 8 (0.83) \\
7 & - & 21 (1.20) & 18 (1.57) & 8 (0.87) \\
8 & - & 21 (1.25) & 18 (1.59) & 8 (0.90) \\
\bottomrule
\end{tabular}
\end{table}

\Cref{tab:helmholtz-results-precond} shows the same trend: the classical V-cycle preconditioner fails for $J\ge6$, while the learnable V-cycle remains effective.
With FMG, the learnable preconditioner gives the best wall clock times for all reported $J$.

\smallskip
\Cref{fig:helmholtz-transfer-overview} summarizes the pre-training and weight-transfer strategies used for the Helmholtz problem in the next three subsections.
It also previews the experiment settings and transfer paths analysis.

\begin{figure}[htbp]
	\centering
	\includegraphics[width=0.99\textwidth]{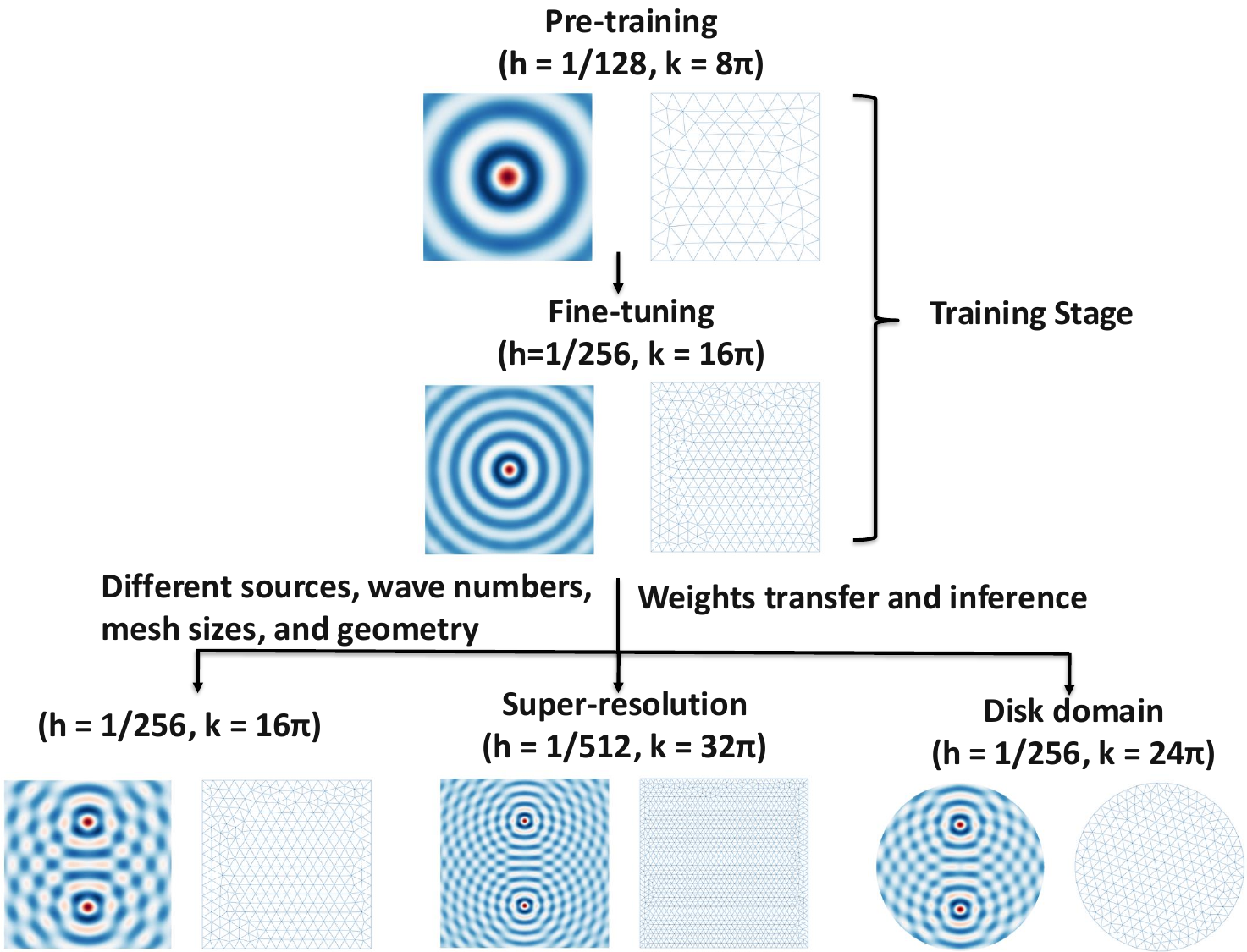}
	\caption{Illustration of the pre-training and weight transfer strategies for the Helmholtz equation. The imaginary part of the solution and the grids are shown for different wave numbers.}
	\label{fig:helmholtz-transfer-overview}
\end{figure}

\subsubsection{Performance enhancement via pre-training strategies}\label{subsec:pretraining}
For the Helmholtz equation, one-shot learning gives a smaller gain than in the Darcy case because the problem is indefinite and the model has more trainable smoother parameters.
We therefore use pre-training: train first on a smaller problem with lower wave number, then use those parameters to initialize the larger target problem.

\smallskip
We first pre-train the model on the Helmholtz equation \cref{eq:helmholtz} with a wave number $k = 8\pi$ and a mesh size $h_{\min} = \frac{1}{128}$.
The number of epochs is set to 100 with no early stopping.
Then, we use the learned parameters as the initial parameters for training on the problem with wave number $k = 16\pi$ and mesh size $h_{\min} = \frac{1}{256}$.
The trained model is then tested as a preconditioner for the FMG method combined with the FGMRES method.
The results for the FGMRES solver are summarized in \cref{tab:pretraining_comparison_gmres}.

\begin{table}[htbp]
\centering
\small
\caption{Comparison of iteration counts and wall clock time (in seconds) for $k = 16\pi, h_{\min} = \frac{1}{256}$ across different levels $J$, with the FMG preconditioner.
Pre-training means that the weights are transferred from the problem with $k = 8\pi, h_{\min} = \frac{1}{128}$, then used to initialize the model.}
\label{tab:pretraining_comparison_gmres}
\begin{tabular}{c cc c@{\hspace{0pt}}c c@{\hspace{0pt}}c}
	\toprule
	\multirow{2}{*}{$J$} & \multicolumn{2}{c}{Classical MG} & \multicolumn{2}{c}{Learnable MG (No pre-train)} & \multicolumn{2}{c}{Learnable MG (Pre-train)} \\
	\cmidrule(lr){2-3} \cmidrule(lr){4-5} \cmidrule(lr){6-7}
	& Steps & Time & \multicolumn{1}{c@{\hspace{20pt}}}{\hspace{20pt}Steps} & Time & \multicolumn{1}{c@{\hspace{20pt}}}{\hspace{20pt}Steps} & Time \\
	\midrule
	5 & 11 & 0.85 & 9 & 0.87 & 7 & 0.70 \\
	6 & 19 & 1.54 & 8 & 0.83 & 8 & 0.82 \\
	7 & 18 & 1.57 & 8 & 0.87 & 8 & 0.87 \\
	8 & 18 & 1.59 & 8 & 0.90 & 6 & 0.69 \\
	\bottomrule
\end{tabular}
\end{table}

The pre-trained model improves the FMG-preconditioned FGMRES results at low additional cost.
Although not every pre-training run improves every case, it is beneficial overall; for example, at $J=5$ the learnable solver outperforms the classical solver only after pre-training.

\subsubsection{Zero-shot super-resolution via weight transfer}
To test the generalization ability further,
we load the weights of the previous model trained in \cref{subsec:pretraining} on a higher wave number problem directly without fine-tuning.
The wave number is set to $k = 32\pi$ and the mesh size is $h_{\min} = \frac{1}{512}$.
We compare the performance of the learnable multigrid solver with the classical multigrid solver as preconditioners for the FMG method combined with the FGMRES method.
The results are summarized in \cref{tab:full_weight_transfer_comparison}.

\begin{table}[htbp]
\centering
\small
\caption{Comparison of iteration counts and wall clock time (in seconds) with weight transfer for $k = 32\pi, h_{\min} = \frac{1}{512}$ across different levels $J$, 
with the FMG preconditioner. The weights are loaded from the pre-trained model with $k = 16\pi, h_{\min} = \frac{1}{256}$ from \cref{subsec:pretraining}.
The -- notation indicates that the solver did not converge within the maximum number of iterations.}
\label{tab:full_weight_transfer_comparison}
\begin{tabular}{ccccc}
	\toprule
	\multirow{2}{*}{$J$} & \multicolumn{2}{c}{Classical MG} & \multicolumn{2}{c}{Learnable MG} \\
	\cmidrule(lr){2-3} \cmidrule(lr){4-5}
	& Steps & Time & \multicolumn{1}{c@{\hspace{20pt}}}{\hspace{20pt}Steps} & Time \\
	\midrule
	5 & 25 & 6.39 & 14 & 4.61 \\
	6 & - & - & 17 & 5.99 \\
	7 & - & - & 19 & 6.57 \\
	8 & - & - & 11 & 3.96 \\
	\bottomrule
\end{tabular}
\end{table}

\Cref{tab:full_weight_transfer_comparison} shows that weight transfer also generalizes well: the learnable preconditioner converges in all cases and significantly outperforms the classical one.
The best result occurs at $J=8$, suggesting that the additional transferred parameters can help at the higher wave number.

\subsubsection{The partial weight transfer strategy for generalization to different wave numbers, geometries, and the number of levels}
Transferring weights across wave numbers, geometries, and level counts requires handling two mismatches: the smoother type at a level may change because of the condition $kh_\ell\le0.5$, and the target hierarchy may have more levels.
Our partial transfer rule is simple.
Levels with matching smoother type receive the corresponding pre-trained weights; levels with mismatched type use the classical non-learnable smoother; extra coarse levels copy the deepest compatible source weights; the shared $\Phi_R$ and $\Phi_P$ are always transferred.

\smallskip
In this subsection, we set the target domain as a unit disk centered at the origin with radius $0.5$.
The wave number is set to $k = 24\pi$ and the mesh size is $h_{\min} = \frac{1}{256}$.
The weights are loaded from the pre-trained model with $k = 16\pi$, $h_{\min} = \frac{1}{256}$ and $J = 5$,
without any fine-tuning on the target problem.

\smallskip
The two mismatches that arise when transferring the pre-trained model ($J=5$, $k=16\pi$) to a deeper, higher-wave-number target ($J=6$, $k=24\pi$) are illustrated in \cref{tab:partial_transfer_rule}.
First, because the wave number increases while the mesh size is fixed, the threshold $kh_\ell \leq 0.5$ is crossed at a finer level: 
the learned Jacobi smoother at source level $\ell = 2$ no longer matches the target, where a GMRES smoother is required.
A direct copy of the learned Jacobi weights into a GMRES smoother is therefore not meaningful.
Second, the target hierarchy has one extra coarse level that has no counterpart in the source model.
Our partial transfer rule resolves both issues without any fine-tuning: levels whose smoother type matches receive the corresponding learned weights, the mismatched level falls back to the classical non-learnable GMRES smoother, the extra coarse level copies the deepest compatible source weights, and the shared $\Phi_R$ and $\Phi_P$ are always transferred.

\begin{table}[htbp]
\centering
\small
\caption{Illustration of the partial weight transfer strategy when transferring the pre-trained model ($J=5$, $k=16\pi$, $h_{\min}=\frac{1}{256}$) from \cref{subsec:pretraining} to the target problem ($J=6$, $k=24\pi$, $h_{\min}=\frac{1}{256}$). Level $\ell = 1$ is the finest grid, and the smoother type at each level follows the condition $kh_\ell \leq 0.5$. 
The shared convolutions $\Phi_R$ and $\Phi_P$ are always transferred.}
\label{tab:partial_transfer_rule}
\begin{tabular}{c cc cc l}
\toprule
\multirow{2}{*}{$\ell$} & \multicolumn{2}{c}{Source} & \multicolumn{2}{c}{Target} & \multirow{2}{*}{Transfer action} \\
\cmidrule(lr){2-3} \cmidrule(lr){4-5}
 & $kh_\ell$ & Smoother & $kh_\ell$ & Smoother & \\
\midrule
1 & 0.20 & Jacobi\ (learned) & 0.30 & Jacobi\ & Transfer weights \\
2 & 0.39 & \textbf{Jacobi}\ (learned) & 0.59 & \textbf{GMRES} & Use classical GMRES \\
3 & 0.79 & GMRES (learned) & 1.18 & GMRES & Transfer weights \\
4 & 1.57 & GMRES (learned) & 2.36 & GMRES & Transfer weights \\
5 & 3.14 & GMRES (learned) & 4.71 & GMRES & Transfer weights \\
6 & --   & \textbf{--}              & 9.42 & \textbf{GMRES} & Copy from $\ell=5$ \\
\bottomrule
\end{tabular}
\end{table}

\smallskip
The inference results with partial weight transfer are summarized in \cref{tab:circle_weight_transfer_comparison}.
It shows that partial transfer improves both iteration counts and wall clock time for all $J$.
The nearly constant iteration counts indicate that the transfer rule handles level mismatch effectively.

\begin{table}[htbp]
\centering
\small
\caption{Comparison of iteration counts and wall clock time (in seconds) with partial weight transfer for $k = 24\pi, h_{\min} = \frac{1}{256}$ across different levels $J$, with the FMG preconditioner.
The weights are loaded from the pre-trained model with $k = 16\pi, h_{\min} = \frac{1}{256}$ and \textbf{$J = 5$} from \cref{subsec:pretraining} by using the partial weight transfer strategy.}
\label{tab:circle_weight_transfer_comparison}
\begin{tabular}{ccccc}
	\toprule
	\multirow{2}{*}{$J$} & \multicolumn{2}{c}{Classical MG} & \multicolumn{2}{c}{Learnable MG} \\
	\cmidrule(lr){2-3} \cmidrule(lr){4-5}
	& Steps & Time & \multicolumn{1}{c@{\hspace{20pt}}}{\hspace{20pt}Steps} & Time \\
	\midrule
	5 & 19 & 1.97 & 14 & 1.67 \\
	6 & 20 & 2.12 & 14 & 1.71 \\
	7 & 20 & 2.18 & 14 & 1.77 \\
	8 & 20 & 2.29 & 14 & 1.86 \\
	\bottomrule
\end{tabular}
\end{table}

\section{Conclusion}
\label{sec:conclusion}
In this paper, we proposed a learnable multigrid framework using graph convolutions.
The learnable multigrid components, including the smoother, restriction, and prolongation operators,
are enhanced with graph convolution operators to improve the overall performance of the multigrid solver.
The framework is general for both geometric and algebraic multigrid methods.

\smallskip
The proposed learnable multigrid solver is lightweight in terms of the number of learnable parameters.
The graph convolution operators used in the experiments have only a few learnable parameters.
The graph convolution operators are embedded into the multigrid components,
not replacing them entirely, nor combining them into a black-box neural network.
This ensures that the learnable multigrid solver retains the interpretability and theoretical foundation of the classical multigrid method.

\smallskip
We conduct numerical experiments on both the Darcy equation with heterogeneous coefficients and the Helmholtz equation with impedance boundary conditions to demonstrate the effectiveness of the proposed learnable multigrid solver.
For the Darcy equation, the method shows strong generalization capabilities across different conductivity fields and mesh sizes.
For the Helmholtz equation, which is challenging due to its indefiniteness and oscillatory nature,
the results demonstrate that the learnable multigrid solver improves convergence compared to classical multigrid methods.
Furthermore, we introduce pre-training and weight transfer strategies to enhance scalability and generalization to different wave numbers, geometries, and mesh sizes.
These strategies allow the model to be effectively applied to larger-scale problems and higher wave numbers with reduced training costs.

\smallskip 
Possible future directions include exploring more advanced graph convolution operators tailored for PDEs,
investigating adaptive strategies for selecting smoothing steps based on problem characteristics,
multi-stage pre-training strategies to enhance generalization across a wider range of problems.


\bibliographystyle{siamplain}
\bibliography{references}

\end{document}